\RequirePackage[l2tabu,orthodox]{nag}
\documentclass[reqno,12pt,oneside]{amsart}
\usepackage
           {geometry}
\usepackage{ifxetex,ifluatex}
\newif\ifxetexorluatex
\ifxetex
  \xetexorluatextrue
\else
  \ifluatex
    \xetexorluatextrue
  \else
    \xetexorluatexfalse
  \fi
\fi

\ifxetexorluatex
  \usepackage{fontspec}           
  \usepackage{unicode-math}       
  \usepackage{polyglossia}        
  \setmainlanguage{english}
\else
  \usepackage[utf8]{inputenc}     
  \usepackage[english]{babel}     

  \usepackage{amsfonts}
  \usepackage{amssymb}
  
  \usepackage{lmodern}
  \usepackage[T1]{fontenc}
\fi
\usepackage{booktabs}   
\usepackage[inline]{enumitem} 
\usepackage{xspace}     
\usepackage[svgnames]{xcolor}     
\usepackage{amsmath}    
\usepackage{mathtools}  
\usepackage{stmaryrd} 

\usepackage{graphics}
\usepackage{etoolbox}
\usepackage{blkarray}
\usepackage{tikz-cd}
\usepackage[font=small,labelfont=bf]{caption}

\definecolor{linkblue}{RGB}{1,1,190}
\definecolor{citegreen}{RGB}{1,190,1}
\usepackage[linkcolor=linkblue
           ,citecolor=citegreen
           ,ocgcolorlinks        
           ,bookmarksopen=True,  
           ]{hyperref}
\usepackage[ocgcolorlinks]{ocgx2} 

\makeatletter
  \AtBeginDocument{
    \hypersetup{
      pdftitle  = {\@title},
      pdfauthor = {\authors}     
    }
  }
\makeatother

\usepackage{amsthm}

\usepackage[capitalize    
           ]{cleveref}

\theoremstyle{definition}
\newtheorem {definition}{Definition}[section]

\theoremstyle{plain}
\newtheorem {theorem}[definition]{Theorem}
\crefname   {theorem}{Theorem}{Theorems}
\newtheorem*{theorem*}{Theorem}
\newtheorem {lemma}[definition]{Lemma}
\newtheorem {proposition}[definition]{Proposition}

\newtheorem {corollary}[definition]{Corollary}

\theoremstyle{remark}
\newtheorem {remark}[definition]{Remark}
\newtheorem {example}[definition]{Example}
\newtheorem*{example*}{Example}
\crefname   {example}{Example}{Examples}
\crefname   {figure}{Figure}{Figures}

\AtBeginEnvironment{smallremark}{\footnotesize}

\makeatletter

\newcommand{\sc@lettershortcut}[3]{%
  \expandafter\providecommand\csname #2#3\endcsname{#1{#3}}%
}

\newcommand{\sc@shortcuts}[3]{%
  \count@=0
  \loop
  \advance\count@ 1
  \edef\tmp@{%
    \noexpand\sc@lettershortcut\unexpanded{{#1}}{#2}{#3\count@}
  }
  \tmp@
  \ifnum\count@<26
  \repeat
}

\newcommand{\defshortcuts}[2]{\sc@shortcuts{#1}{#2}{\@alph}}

\newcommand{\defShortcuts}[2]{\sc@shortcuts{#1}{#2}{\@Alph}}

\makeatother
\defShortcuts{\mathbb}{b}
\defShortcuts{\mathcal}{c}
\defShortcuts{\mathfrak}{f}
\defShortcuts{\mathsf}{s}
\defshortcuts{\mathfrak}{f}
\defshortcuts{\mathsf}{s}

\def\rfop{*}
\makeatletter
\newcommand\rigidfactorization[2][]{%
  \def\rf@delim{\rfop}
  \newif\ifrf@notfirst
  #1
  \@for\next:=#2\do{%
    \ifrf@notfirst
      \rf@delim
    \fi
    \rf@notfirsttrue
    \next
  }%
}
\makeatother
\newcommand\rf\rigidfactorization
\ifxetexorluatex
  
\else
  
\fi

\DeclarePairedDelimiter{\length}{\lvert}{\rvert}
\DeclarePairedDelimiter{\abs}{\lvert}{\rvert}
\DeclarePairedDelimiter{\card}{\lvert}{\rvert}

\DeclareMathOperator{\End}{End}

\DeclareMathOperator{\GL}{GL}

\DeclareMathOperator{\Hom}{Hom}

\DeclareMathOperator{\chr}{char}

\DeclareMathOperator{\supp}{supp}

\newcommand{\defit}{\emph}

\setlist[enumerate,1]{label=\textup{(\arabic*)}, ref=\textup{(}\arabic*\textup{)}, leftmargin=0.75cm}
\setlist[enumerate,2]{label=\textup{(}\roman*\textup{)}, ref=\textup{(}\roman*\textup{)}}

\newlist{equivenumerate}{enumerate}{1}
\setlist[equivenumerate,1]{%
  label=\textup{(\alph*)},
  ref=\textup{(}\alph*\textup{)},
  leftmargin=0.75cm
}

\newlist{equivenumerate*}{enumerate*}{1}
\setlist*[equivenumerate*,1]{%
  label=\textup{(\alph*)},
  ref=\textup{(}\alph*\textup{)},
  leftmargin=0.75cm
}

\newlist{propenumerate}{enumerate}{1}
\setlist[propenumerate,1]{%
  label=\textup{(\roman*)},
  ref=\textup{(}\roman*\textup{)},
  leftmargin=0.75cm
}

{\end{description}}

\makeatletter
\def\sr@stripleadingcol::#1{#1}
\def\sr@dosubref#1#2:#3 #4{\if\relax#3\relax%
  \def\first{\sr@stripleadingcol #4}%
  #1{\first}\ref{\first:#2}%
\else%
  \sr@dosubref#1#3 {#4:#2}%
\fi}%

\newcommand{\subref}[1]{\sr@dosubref\cref#1: :\relax}
\newcommand{\Subref}[1]{\sr@dosubref\Cref#1: :\relax}
\makeatother
\AtBeginDocument{\renewcommand{\setminus}{\smallsetminus}}

\usetikzlibrary{positioning,automata}
\tikzset{>=latex[round]}

\newcommand{\llangle}{\langle\!\langle}
\newcommand{\rrangle}{\rangle\!\rangle}

\newcommand{\defi}{\textsf}

\usepackage{array}
\usepackage{arydshln}
\DeclareMathOperator{\PGL}{PGL}

\usepackage{microtype}

\title{Noncommutative rational Pólya series}
\author{Jason Bell}
\address{Department of Pure Mathematics, University of Waterloo, Waterloo, ON, Canada N2L 3G1}
\email{jpbell@uwaterloo.ca}

\author{Daniel Smertnig}
\email{dsmertni@uwaterloo.ca}

\newcommand{\irr}{\cZ}
\usepackage{bbold}
\newcommand{\charser}[1]{\mathbb{1}_{{#1}}}

\ifxetexorluatex
  
\else
  
  \fi


\begin{document}

\begin{abstract}
  A (noncommutative) Pólya series over a field $K$ is a formal power series whose nonzero coefficients are contained in a finitely generated subgroup of $K^\times$.
  We show that rational Pólya series are unambiguous rational series, proving a 40 year old conjecture of Reutenauer.
  The proof combines methods from noncommutative algebra, automata theory, and number theory (specifically, unit equations).
  As a corollary, a rational series is a Pólya series if and only if it is Hadamard sub-invertible.
  Phrased differently, we show that every weighted finite automaton taking values in a finitely generated subgroup of a field (and zero) is equivalent to an unambiguous weighted finite automaton.
\end{abstract}

\maketitle

\section{Introduction and main results}

Let $K$ be a field.
A univariate formal power series
\[
  S=\sum_{n \ge 0} s(n) x^n \in K\llbracket x \rrbracket
\]
is a \defi{rational series} if it is the power series expansion of a rational function at $0$.
Necessarily, this rational function does not have a pole at $0$.
Equivalently, the coefficients of a rational series satisfy a \defi{linear recurrence relation}, that is, there exist $\alpha_1$, $\ldots\,$,~$\alpha_m \in K$ such that
\[
  s(n+m) = \alpha_1 s(n+m-1) + \cdots + \alpha_m s(n) \quad\text{for all $n \ge 0$}.
\]
Pólya \cite{polya21} considered arithmetical properties of rational series over $K=\bQ$, and characterized the univariate rational series whose coefficients are supported at finitely many prime numbers.
This was later extended to number fields by Benzaghou \cite[Chapitre 5]{benzaghou70}, and to arbitrary fields, in particular fields of positive characteristic, by Bézivin \cite{bezivin87}.
Ultimately, they proved the following theorem.

We call a rational series $S \in K\llbracket x \rrbracket$ a \defi{Pólya series} if there exists a finitely generated subgroup $G \le K^\times$, such that all coefficients of $S$ are contained in $G \cup \{0\}$.

\begin{theorem}[Pólya; Benzaghou; Bézivin] \label{t:univariate}
  Let $K$ be a field, let $S=P/Q \in K(x)$ be a rational function with $Q(0) \ne 0$, and let
  \[
    S = \sum_{n=0}^\infty s(n) x^n \in K\llbracket x\rrbracket
  \]
  be the power series expansion of $S$ at $0$.
  Suppose that $S$ is a Pólya series.
  Then there exist a polynomial $T \in K[x]$, $d \in \bZ_{\ge 0}$, and for each $r \in [0,d-1]$ elements $\alpha_r \in K$, $\beta_r \in K^\times$ such that
  \[
    S = T + \sum_{r=0}^{d-1} \frac{\alpha_r x^r}{1-\beta_r x^d}.
  \]
  Equivalently, there exists a finite set $F \subseteq \bZ_{\ge 0}$ such that
  \[
    s(kd+r) = \alpha_r \beta_r^k \qquad\text{for all } k \ge 0 \text{ and } r \in [0,d-1] \text{ with $kd+r \not\in F$}.
  \]
\end{theorem}

The converse of the previous theorem, that every series with such coefficients is a Pólya series, is of course trivial.

Let $R$ be a (commutative) domain.
A noncommutative formal power series $S=\sum_{w \in X^*} S(w) w \in R\llangle X \rrangle$ is \defi{rational} if it can be obtained from noncommutative polynomials in $R\langle X\rangle$ by successive applications of addition, multiplication, and the star operation $S^* = (1-S)^{-1} = \sum_{n \ge 0} S^n$ (if $S$ has zero constant coefficient).
See \cref{sec:preliminaries} below for formal definitions, and the book by Berstel--Reutenauer \cite{berstel-reutenauer11} for more background.
Extending the correspondence between univariate rational series and linear recurrence relations to the noncommutative setting, a theorem of Schützenberger shows that $S$ is rational if and only if it has a linear representation, or, equivalently, is recognized by a weighted finite automaton.

The definition of Pólya series extends to this noncommutative setting:
A rational series $S \in R\llangle X \rrangle$ is a \defi{Pólya series} if its nonzero coefficients are contained in a finitely generated subgroup $G\le K^\times$ of the quotient field $K$ of $R$.
Noncommutative rational Pólya series were first studied by Reutenauer in 1979 \cite{reutenauer79}.
Reutenauer introduced the notion of an unambiguous rational series and conjectured that these should be precisely the rational Pólya series \cite{reutenauer79}; see also \cite{reutenauer80}, \cite[\S 6]{reutenauer96} and \cite[p.233, Open Problem 4]{berstel-reutenauer11}.
He proved many equivalent characterizations of unambiguous rational series, for instance, showing that they are precisely the ones being recognized by an unambiguous weighted finite automaton.
The conjecture that rational Pólya series are unambiguous however so far remained open.

The goal of the present paper is to prove this conjecture.
We also recover a new proof of \cref{t:univariate} as a special case of our more general theorem, and give a characterization of those rational series recognized by a deterministic weighted finite automaton.
Moreover, this resolves all three conjectures in \cite[Chapitre 6]{reutenauer80}.

A rational series is \defi{unambiguous} if it can be obtained from noncommutative polynomials and the operations of addition, multiplication, and the star operation $S^* = (1-S)^{-1} = \sum_{n \ge 0} S^n$ in such a way that, in these operations, one never forms a sum of two nonzero coefficients. (This is defined more formally in \cref{def:unambiguous-rat} below; in \cref{sec:preliminaries} we also recall the definitions of rational series and (unambiguous) weighted automata.)
A formal series $S \in \bZ\llangle X \rrangle$ is \defi{linearly bounded} if there exists $C \ge 0$ such that $\abs{S(w)} \le C\length{w}$ for all nonempty words $w \in X^*$.

Let $R$ be a (commutative) domain and $K$ its quotient field.
An element $a \in K$ is \defi{almost integral} over $R$ if there exists $0 \ne c \in R$ such that $ca^n \in R$ for all $n \ge 0$.
A domain $R$ is \defi{completely integrally closed} if it contains all such almost integral elements.
We are mostly interested in the cases where $R=K$ is a field or $R=\bZ$.
In general, one cannot relax the completely integrally closed condition in the following theorem to integrally closed---see \cref{rem:cic} and the example following it.

\begin{theorem} \label{t:main}
  Let $R$ be a completely integrally closed domain with quotient field $K$.
  Let $X$ be a finite non-empty set, and let $S \in R \llangle X \rrangle$ be a rational series.
  Then the following statements are equivalent.
  \begin{equivenumerate}
  \item \label{main:polya} $S$ is a Pólya series.
  \item \label{main:unam-auto}$S$ is recognized by an unambiguous weighted finite automaton with weights in $R$.
  \item \label{main:unam-rat} $S$ is unambiguous \textup{(}over $R$\textup{)}.
  \item \label{main:formula} There exist $\lambda_1$, $\ldots\,$,~$\lambda_k \in R \setminus \{0\}$, linearly bounded rational series $a_1$, $\ldots\,$,~$a_k \in \bZ\llangle X\rrangle$, and a rational language $\cL \subseteq X^*$ such that $\supp(a_i) \subseteq \cL$ for all $i \in [1,k]$ and
    \[
      S(w) =
      \begin{cases}
        \lambda_1^{a_1(w)} \cdots \lambda_k^{a_k(w)} & \text{if $w \in \cL$},\\
        0 & \text{if $w \not\in \cL$}.
      \end{cases}
    \]
  \item \label{main:hadamard} $S$ is Hadamard sub-invertible, that is, the series
    \[
      \sum_{w \in \supp(S)} S(w)^{-1} w \in K\llangle X \rrangle
    \]
    is a rational series \textup{(}over $K$\textup{)}.
  \end{equivenumerate}
\end{theorem}

The `hard' part of this theorem is showing \ref{main:polya}$\,\Rightarrow\,$\ref{main:unam-auto} in the case where $R=K$ is a field.
It involves the use of finiteness results on unit equations in characteristic $0$ and a recent extension of Derksen--Masser to positive characteristic \cite{derksen-masser12}.
We also make use of a new invariant associated to a linear representation, its linear hull.
The linear hull also allows a characterization of determinizable weighted automatons, see \cref{t:determinizable} below.

The other implications of \cref{t:main} are comparatively straightforward and are largely known.
The equivalence \ref{main:unam-auto}$\,\Leftrightarrow\,$\ref{main:unam-rat} was first noted by Reutenauer \cite[Chapitre VI, Théorème 1]{reutenauer80}. The implications \ref{main:unam-rat}$\,\Rightarrow\,$\ref{main:hadamard} and \ref{main:hadamard}$\,\Rightarrow\,$\ref{main:polya} are also known \cite[Exercise 3.1 of Chapter 6]{berstel-reutenauer11}, and so once \ref{main:polya}$\,\Rightarrow\,$\ref{main:unam-auto} is shown, the equivalence of \ref{main:polya}, \ref{main:unam-auto}, \ref{main:unam-rat}, and \ref{main:hadamard} is clear. Finally, \ref{main:polya}$\,\Leftrightarrow\,$\ref{main:formula} appears in the proof of \cite[Proposition 4(ii)]{reutenauer79}.
Despite this, we opt to give a self-contained proof of all of \cref{t:main} in the present paper.

Denote by $\mathbb 1_\cL$ the characteristic series of a set $\cL \subseteq X^*$.
By a theorem of Schützenberger \cite[Corollary 9.2.6]{berstel-reutenauer11}, any linearly bounded rational series $a \in \bZ\llangle X \rrangle$ can be expressed as a $\bZ$-linear combination of series of the form $\mathbb 1_\cL$ and $\mathbb 1_\cL \mathbb 1_\cK$ for rational languages $\cL$, $\cK$.
This gives a more explicit description of the series appearing as exponents in \ref{main:formula} of \cref{t:main}.

In computer science the question whether a given weighted automaton is equivalent to a deterministic (sequential) one has received considerable attention; we mention the survey \cite{lombardy-sakarovitch06}.
The question is of theoretical importance but also of practical relevance in natural language processing \cite{mohri97,buchsbaum-giancarlo-westbrook00,mohri-riley17}.
In this context, in contrast to our setting, however $K$ is usually a tropical semiring.
Nevertheless one can also pose these questions for $K$ a field, as is done for instance in \cite[\S5]{lombardy-sakarovitch06}.

When $K$ is a field, the \defi{linear hull} (see \cref{d:linearhull}) allows a characterization of rational series recognized by a deterministic weighted automaton.
See \cref{sec:determinizable} for another characterization, using bounded variation, in the spirit of Mohri \cite{mohri97}.

\begin{theorem} \label{t:determinizable}
  Let $K$ be a field and $X$ a finite non-empty set.
  For a rational series $S \in K\llangle X \rrangle$, the following statements are equivalent.
  \begin{equivenumerate}
  \item \label{det:det} $S$ is recognized by a deterministic weighted automaton.
  \item \label{det:dim} If $(u,\mu,v)$ is a minimal linear representation of $S$, then its linear hull has dimension at most $1$.
  \end{equivenumerate}
\end{theorem}

Rephrasing \cref{t:determinizable}, a weighted automaton (with weights in a field $K$) is equivalent to a deterministic one if and only if \ref{det:dim} holds.
We also obtain the following (already known) corollary.

\begin{corollary} \label{c:finite}
  If $S \in K\llangle X \rrangle$ is a rational series whose coefficients take only \emph{finitely} many values, then $S$ is recognized by a deterministic weighted automaton.
  In particular, if $K$ is a finite field, then every rational series over $K$ is recognized by a deterministic weighted automaton.
\end{corollary}

\subsection*{Notation}
Throughout the paper, let $R$ be a (commutative) domain with quotient field $K$.
Often we will be concerned only with the case where $R=K$ is a field.
Let $X$ be a finite non-empty set.
Let $G \le K^\times$ be a finitely generated subgroup, and set $G_0=G \cup \{0\}$.
When considering Pólya series $S$, we will assume that $G$ is such that $G_0$ contains all coefficients of $S$.

\subsection*{Outline}
The paper is organized as follows.
In \cref{sec:preliminaries} we recall necessary background on rational series.
In \cref{sec:linearhull} we introduce a useful topology and the notion of a linear hull; we also make a first crucial reduction in \cref{l:repr-directsum} and obtain \cref{c:finite}.
In \cref{sec:unit-equations} we use unit equations to prove a key lemma, with the majority of the work dedicated to dealing with positive characteristic.
In \cref{sec:linearhull,sec:unit-equations} we restrict to the case where $R=K$ is a field.
Now we can prove \cref{t:main} over fields: in \cref{sec:polya-unam-auto} we prove the hard direction \ref{main:polya}$\,\Rightarrow\,$\ref{main:unam-auto}.
In \cref{sec:unam-auto-unam-rat} we show \ref{main:unam-auto}$\,\Leftrightarrow\,$\ref{main:unam-rat}, in \cref{sec:unam-rat-formula} we show \ref{main:unam-rat}$\,\Rightarrow\,$\ref{main:formula}.
The implications \ref{main:unam-rat}$\,\Rightarrow\,$\ref{main:hadamard} and \ref{main:hadamard}$\,\Rightarrow\,$\ref{main:polya} are shown in \cref{sec:hadamard}.
In \cref{sec:all-together} we put all these pieces together and extend the main implication \ref{main:polya}$\,\Rightarrow\,$\ref{main:unam-auto} from fields to completely integrally closed domains, then prove \cref{t:main,t:univariate}.
In \cref{sec:determinizable} we conclude the proof of \cref{t:determinizable}.

\subsection*{Acknowledgments}
We thank Christophe Reutenauer for many helpful comments on an earlier version of this manuscript, and Daniela Petrisan, Amaury Pouly, as well as Jacques Sakarovitch for fruitful discussions on topics related to the paper.
We are grateful to the anonymous referee for their careful reading and helpful comments; in particular their request for the addition of illustrative examples will surely be appreciated by readers.

Bell was supported by NSERC grant RGPIN-2016-03632.
Smertnig was supported by the Austrian Science Fund (FWF) project J4079-N32.

\section{Preliminaries: rational series, linear representations, and weighted automata}
\label{sec:preliminaries}
We briefly recall the definitions of (noncommutative) rational series, linear representations, and weighted automata and how they relate to each other.
We largely follow the notation and terminology from \cite{berstel-reutenauer11}.

Let $X^*$ denote the free monoid on the alphabet $X$.
For a (noncommutative) formal power series $S \in R\llangle X \rrangle$ and a word $w \in X^*$, we write $S(w)$ for the coefficient of $w$, that is
\[
  S = \sum_{w \in X^*} S(w) w.
\]
The \defi{support} of $S$ is $\supp(S) = \{\, w \in X^* : S(w) \ne 0 \,\}$.

The ring of \defi{rational series} in $X$ is the smallest subring of the power series ring $R\llangle X \rrangle$ that contains the noncommutative polynomials $R\langle X \rangle$ and is closed under addition, multiplication, and the partial operation
\[
  S \mapsto S^* \coloneqq (1-S)^{-1} = \sum_{n \ge 0} S^n
\]
whenever $S$ has zero constant coefficient.

If $X=\{x\}$ is a singleton, then $R \llangle X \rrangle = R\llbracket x \rrbracket$.
One can easily check that $S \in R\llbracket x \rrbracket$ is a rational series if and only if there exist polynomials $P$,~$Q \in R[x]$ with $Q(0) =1$ such that $S=P/Q$ \cite[Proposition 6.1.1]{berstel-reutenauer11}.
In other words, $S$ is rational if and only if it is the power series expansion, at the point $0$, of a rational function not having a pole at $0$.
If $R=K$ is a field, it is also well-known (and not hard to check) that this is the case if and only if the coefficients of $S$ satisfy a linear recurrence relation.
Equivalently, there exist vectors $u \in K^{1 \times n}$, $v \in K^{n \times 1}$, and a matrix $A \in K^{n \times n}$ such that $S(x^i) = u A^i v$ for every $i \ge 0$.

A fundamental theorem of Schützenberger extends this description to multivariate noncommutative rational series.
A \defi{linear representation} of rank (dimension) $n$ is a triple $(u,\mu,v)$ where $u\in R^{1 \times n}$ and $v \in R^{n \times 1}$ are vectors, and $\mu \colon X^* \to R^{n \times n}$ is a monoid homomorphism from the free monoid $X^*$ to multiplicative monoid of $n \times n$-matrices.
Schützenberger showed that $S \in R\llangle X \rrangle$ is rational if and only if there exists a linear representation $(u,\mu,v)$ such that $S(w) = u \mu(w) v$ for every $w \in X^*$; see \cite[Theorem 1.7.1]{berstel-reutenauer11}.

Suppose $R=K$ is a field.
A linear representation is \defi{minimal} if the dimension $n$ is minimal among all possible linear representations of $S$.
This is the case if and only if the span of $u\mu(X^*) = \{\, u\mu(w)  : w \in X^* \,\}$ is $K^{1\times n}$ and the span of $\mu(X^*)v$ is $K^{n \times 1}$.

There is another, graph-theoretical, way to view linear representations over the domain $R$ that will come in handy.
A \defi{weighted \textup{(}finite\textup{)} automaton} $\cA=(Q,I,E,T)$ over the alphabet $X$ with weights in $R$ consists of a finite set of \defi{states} $Q$ and three maps
\[
  I \colon Q \to R,\quad E \colon Q \times X \times Q \to R,\quad T \colon Q \to R.
\]
A triple $(p,x,q) \in Q \times X \times Q$ is an \defi{edge} if $E(p,x,q) \ne 0$.
More specifically, we say that there is an edge from $p$ to $q$ labeled by $x$ and with weight $E(p,x,q)$.
A state $p \in Q$ is \defi{initial} if $I(p) \ne 0$ and \defi{terminal} if $T(p) \ne 0$.

A \defi{path} is a sequence of edges
\[
  P = (p_0,x_1,p_1)(p_1,x_2,p_2)\cdots (p_{l-1},x_l,p_l).
\]
Its \defi{weight} is $E(P) = \prod_{i=1}^l E(p_{i-1},x_i,p_i)$ and its \defi{label} is the word $x_1 \cdots x_l \in X^*$.
The path is \defi{accepting} if $p_0$ is an initial state and $p_l$ is a terminal state.
The automaton is \defi{trim} if every state lies on an accepting path.

The series $S \in R\llangle X \rrangle$ is \defi{recognized} by $\cA$ if
\begin{equation}\label{e:auto-recog}
  S(w) = \sum_{\substack{p_0,p_1,\ldots,p_l \in Q\\ w=x_1\cdots x_l,\, x_i \in X}} I(p_0)E(p_0,x_1,p_1) \cdots E(p_{l-1},x_l,p_l) T(p_l).
\end{equation}
Thus, the coefficient $S(w)$ is obtained by summing the weights of all accepting paths labeled by the word $w$, weighing each path by initial/terminal weights.
Two automata are \defi{equivalent} if they recognize the same series.
Obviously every weighted automaton is equivalent to a trim one.

There is an easy correspondence between linear representations and weighted automata.
Explicitly, the weighted automaton associated to a linear representation $(u,\mu,v)$ is given by $Q=[1,n]$, with $I(k)=u_k$, with $T(k)=v_k$, and $E(k,x,l)=\mu(x)_{k,l}$; here the subscripts denote the corresponding coordinates of the vectors $u$ and $v$, respectively the entries of the matrix $\mu(x)$.
Conversely, for a weighted automaton, one may without loss of generality assume $Q=[1,n]$, and then the correspondence above yields a linear representation (a different labeling of the states gives a conjugate linear representation, corresponding to a permutation of the basis vectors).
A series is recognized by the weighted finite automaton if and only if it is recognized by the associated linear representation.
Hence series recognized by automata and series with linear representations are the same, and by Schützenberger's Theorem coincide with rational series.

\begin{definition}
  Let $\cA$ be a weighted automaton.
  Then $\cA$ is \defi{unambiguous} if each $w \in X^*$ labels at most one accepting path.
  It is \defi{deterministic} (or \defi{sequential}) if
  \begin{itemize}
  \item there exists at most one initial state; and
  \item for each $(p,x) \in Q \times X$, there exists at most one $q \in Q$ with $E(p,x,q) \ne 0$.
  \end{itemize}
\end{definition}

Note that for an unambiguous automaton, in the expression \eqref{e:auto-recog} for $S(w)$, at most one summand is nonzero.
Deterministic weighted automata are clearly unambiguous.

\begin{remark}
  For automata without weights (equivalently, weights in the Boolean semiring $\cB=\{0,1\}$ with $1+1=1$), it is well known that every automaton is equivalent to a deterministic one.
  This is no longer true for weighted automata; there exist unambiguous weighted automata that are not equivalent to deterministic ones, and there exist weighted automata that are not equivalent to unambiguous ones.
\end{remark}

\begin{definition}
  A rational series $S \in R\llangle X \rrangle$ is a \defi{Pólya series} if there exists a finitely generated subgroup $G \le K^\times$ of the quotient field $K$ of $R$ such that $S(w) \in G_0=G \cup \{0\}$ for all $w \in X^*$.
\end{definition}

Let $S$,~$T \in R\llangle X \rrangle$ be two series with $\cK = \supp(S)$ and $\cL = \supp(T)$.
The addition $S+T$ is \defi{unambiguous} if $\supp(S) \cap \supp(T) = \emptyset$; the product $ST$ is unambiguous if every $w \in \cK \cL$ has a \emph{unique} expression $w=w_1 w_2$ with $w_1 \in \cK$ and $w_2 \in \cL$; and the star operation $S^*$ is unambiguous if $\cK$ is a \defi{code}, that is, every $w \in \cK^*$ has a \emph{unique} expression $w=w_1\cdots w_k$ with $w_i \in \cK$.

\begin{definition} \label{def:unambiguous-rat}
  The set of \defi{unambiguous rational series} is the smallest subset of $R\llangle X \rrangle$ that contains $R\langle X \rangle$ and is closed under \emph{unambiguous} addition, multiplication, and star operation.
\end{definition}

Note that unambiguous operations are defined in such a way that every coefficient of the resulting series is a product of coefficients of the initial series.
That is, one never forms a sum of two nonzero coefficients.
We thus we have the following.

\begin{lemma}
  Every unambiguous rational series is a Pólya series.
\end{lemma}

\section{The linear hull of a linear representation}
\label{sec:linearhull}

In this section we consider only the case where $R=K$ is a field.
We introduce a topology and a related invariant of a linear representation that will be essential in the proof of the implication \ref{main:polya}$\,\Rightarrow\,$\ref{main:unam-auto} of \cref{t:main}.

\begin{definition}
  For a finite-dimensional vector space $V$, let $\cF(V)$ be the collection of all subsets $Y \subseteq V$ of the form $Y = V_1 \cup \dots \cup V_l$ with $l \in \bZ_{\ge 0}$ and $V_i \subseteq V$ vector subspaces.
\end{definition}

\begin{lemma}
  Every finite-dimensional vector space $V$ has a noetherian topology for which $\cF(V)$ is the collection of closed sets.
\end{lemma}

\begin{proof}
  Set $\cF=\cF(V)$.
  Clearly $V \in \cF$ and $\emptyset  \in \cF$ (with $\emptyset$ represented by the empty union).
  By definition $\cF$ is closed under finite unions.
  To show that $\cF$ is the collection of closed sets of a topology, it remains to verify that $\cF$ is closed under intersections.
  Every $Y \in \cF$ is closed in the Zariski topology (identifying $V$ with $\mathbb A_K^n$ for $n=\dim_K V$), which is noetherian.
  Hence, any intersection is equal to a finite subintersection.
  The claim follows since intersections distribute over unions, and intersections of vector subspaces are again vector subspaces.
  Since every $Y \in \cF$ is Zariski-closed, the topology is noetherian.
\end{proof}

\begin{definition}
  Let $V$ be a finite-dimensional vector space.
  The \defi{linear Zariski topology} on $V$ is the topology whose collection of closed sets is $\cF(V)$.
\end{definition}

If $W \subseteq V$ is a vector subspace, then the subspace topology induced on $W$ is the linear Zariski topology on $W$.
All topological notions occurring in the remainder of the paper will refer to the linear Zariski topology.
We mention \cite[\S II.4.1 and \S II.4.2]{bourbaki:ca72} and \cite[Sections \href{https://stacks.math.columbia.edu/tag/004U}{004U} and \href{https://stacks.math.columbia.edu/tag/0050}{0050}]{stacks-project} as references on irreducible and noetherian topological spaces.

A topological space $X$ is \defi{irreducible} if it is non-empty and $X=Z_1 \cup Z_2$ with $Z_1$,~$Z_2$ closed implies $X=Z_1$ or $X=Z_2$.
A subset $Z \subseteq X$ is an \defi{irreducible component} if it is a maximal irreducible subspace.
By $\irr(X)$ we denote the set of all irreducible components of $X$.
Then $X = \bigcup_{Z \in \irr(X)} Z$.

\begin{lemma}
  The closed irreducible subsets of a vector space $V$ are exactly the
  \begin{enumerate}
  \item vector subspaces of $V$ if $K$ is infinite,
  \item vector subspaces of $V$ of dimension $\le 1$ if $K$ is finite.
  \end{enumerate}
\end{lemma}

\begin{proof}
  If $K$ is infinite, then a vector space cannot be expressed as a finite union of proper vector subspaces.
  It follows that the closed irreducible subsets of $V$ are exactly the vector subspaces.
  If $K$ is finite, we can write every nonzero vector subspace of $V$ as a finite union of one-dimensional vector subspaces.
\end{proof}

The \defi{dimension} of a closed set is the maximal dimension of its irreducible components, with $\dim \emptyset = -\infty$.

We recall the following basic properties, of which we will make use throughout.

\begin{lemma} \label{l:dense}
  Let $Y$ be a topological space.
  \begin{enumerate}
  \item\label{dense:subdense} If $Y$ is irreducible, $Y' \subsetneq Y$ is closed, and $\Omega \subseteq Y$ is dense, then $\overline{\Omega \setminus Y'} = Y$.
  \item\label{dense:continuous} If $Z \subseteq Y$ is irreducible, and $f\colon Y \to Y'$ is continuous, then $f(Z)$ is irreducible.
  \item\label{dense:continuous-component} If $Z \in \irr(Y)$ and $f\colon Y \to Y'$ is continuous, then $f(Z) \subseteq Z'$ for some $Z' \in \irr(Y')$.
  \item\label{dense:finite}  If $Y$ is noetherian, then $\irr(Y)$ is finite.
  \item\label{dense:finitedense} If $\irr(Y)$ is finite, $\Omega \subseteq Y$ is dense, and $Z \in \irr(Y)$, then $\overline{\Omega \cap Z} = Z$.
  \end{enumerate}
\end{lemma}

\begin{proof}
  \ref*{dense:subdense}
  The set $U = Y \setminus Y'$ is non-empty and open and therefore dense in $Y$ by irreducibility \cite[Proposition 1 of \S II.4.1]{bourbaki:ca72}.
  By basic topology, the intersection of a dense subset with an open subset is dense in the open subset.
  Thus $\Omega \setminus Y' = \Omega \cap U$ is dense in $U$, and by transitivity, in $Y$.

  \ref*{dense:continuous} \cite[Proposition 4 of \S II.4.1]{bourbaki:ca72} or \cite[\href{https://stacks.math.columbia.edu/tag/0379}{Lemma 0379}]{stacks-project}.

  \ref*{dense:continuous-component}
  By \ref*{dense:continuous}, the subset $f(Z)$ of $Y'$ is irreducible.
  Every irreducible subset of $Y'$ is a subset of an irreducible component by \cite[Proposition 5 of \S II.4.2]{bourbaki:ca72} or \cite[\href{https://stacks.math.columbia.edu/tag/004W}{Lemma 004W}]{stacks-project}.

  \ref*{dense:finite} \cite[Proposition 10 of \S II.4.2]{bourbaki:ca72} or  \cite[\href{https://stacks.math.columbia.edu/tag/0052}{Lemma 0052}]{stacks-project}.

  \ref*{dense:finitedense} Let $Z_1$, $\ldots\,$,~$Z_m$ be the irreducible components of $Y$, with $Z=Z_1$.
  Then $Y = \bigcup_{i=1}^m Z_i$.
  Let $U \subseteq Z_1$ be relatively open in $Y$ and assume $U \cap \Omega = \emptyset$.
  We have to show $U = \emptyset$.
  Since irreducible components are closed (\cite[Proposition 2 of \S II.4.1]{bourbaki:ca72} or \cite[\href{https://stacks.math.columbia.edu/tag/004W}{Lemma 004W}]{stacks-project}), the set $Z_1 \setminus U$ is closed in $Y$.
  Therefore $\Omega \subseteq (Z_1 \setminus U) \cup Z_2 \cup \cdots \cup Z_m$ implies $Z_1 \subseteq \overline{\Omega} \subseteq (Z_1 \setminus U) \cup Z_2 \cup \cdots \cup Z_m$. The irreducibility of $Z_1$, together with the incomparability of irreducible components, implies $Z_1 \subseteq Z_1 \setminus U$, that is $U=\emptyset$.
\end{proof}

We can now define a key invariant associated to a linear representation.

\begin{definition} \label{d:linearhull}
   Let $(u,\mu,v)$ be a linear representation over the field $K$, and let
   \[
     \Omega \coloneqq u \mu(X^*) = \{\, u \mu(w) : w \in X^* \,\}
   \]
   be the \defi{(left) reachability set}.
  The closed set $\overline \Omega$ is the \defi{\textup{(}left\textup{)} linear hull} of $(u,\mu,v)$.
\end{definition}

Before continuing, we illustrate the linear hull on two examples.

\begin{example} \label{exm:polya}
  Let $K= \bQ$, let $X=\{a,b,c\}$, and define a linear representation $(u,\mu,v)$ by $u=(1,1,1)$, by $v = (1,1,0)^T$, and by
  \begin{align*}
     \mu(a) &= \begin{pmatrix} 2 & 0 & 0 \\ 0 & -2 & 0 \\ 0 & 0 & 3 \end{pmatrix},
    & \mu(b) &= \begin{pmatrix} 0 & 0 & 0 \\ 0 & 0 & 1 \\ 1 & 1 & 0 \end{pmatrix},
    & \mu(c) &= \begin{pmatrix} 0 & 0 & 0 \\ 0 & 0 & 0 \\ 0 & 0 & 5 \end{pmatrix}.
  \end{align*}
  The corresponding automaton is depicted in \cref{fig:polya}.
  It is easy to see that $(u,\mu,v)$ is minimal.
  We claim $\overline{\Omega} = \langle e_1+ e_2, e_3 \rangle \cup \langle e_1-e_2, e_3 \rangle$.
  The inclusion $\subseteq$ follows from $u \in \overline \Omega$ and the fact that $\langle e_1+ e_2, e_3 \rangle \cup \langle e_1-e_2, e_3 \rangle$ is closed under right action by $\mu(x)$ for all $x \in \{a,b,c\}$.
  For the inclusion $\supseteq$, observe $u\mu(a)^n = (2^n, (-1)^n2^n, 3^n)$.
  Therefore $\{\, u\mu(a)^{2n} : n \ge 0\,\} \subseteq \Omega$ is dense in $\langle e_1+e_2, e_3 \rangle$ while $\{\, u\mu(a)^{2n+1} : n \ge 0 \,\} \subseteq \Omega$ is dense in $\langle e_1-e_1, e_3 \rangle$.
  
  One could check that all nonzero coefficients of the series $S = 2 + 2b + 8a^2 + 10cb + 6ab + \cdots$ recognized by $(u,\mu,v)$ are of the form $2^e3^f5^g$ with $e$, $f$,~$g \ge 0$, and that $S$ is therefore a Pólya series.
  However, since we will construct an unambiguous linear representation recognizing the same series, we will see that $S$ is Pólya a posteriori.
\end{example}

\begin{figure}[h]
  \centering
  \begin{tikzpicture}[shorten >=1pt,node distance=3cm,on grid]
    \node[state]   (s1)                {$1$};
    \node[state]   (s2) [right of=s1] {$2$};
    \node[state]   (s3) [below of=s2,yshift=1cm] {$3$};
    \path[->]
    (s3) edge node [above] {$b$} (s1)
    (s3) edge [bend right=15] node [right] {$b$} (s2)
    (s2) edge [bend right=15] node [left,yshift=0.5mm] {$b$} (s3)
    (s1) edge [loop below]   node   {$2a$} (s1)
    (s2) edge [loop left]   node   {$-2a$} (s2)
    (s3) edge [loop below]   node [yshift=-2mm]  {$3a$} (s3)
    (s3) edge [loop right]   node   {$5c$} (s3);

    \path[->]
    ([xshift=0.15cm,yshift=-0.04cm]s1.north) edge +(0,0.75cm)
    +(-0.3cm,0.71cm) edge +(-0.3cm,-0.04cm);
    
    \path[->]
    ([xshift=0.15cm,yshift=-0.04cm]s2.north) edge +(0,0.75cm)
    +(-0.3cm,0.71cm) edge +(-0.3cm,-0.04cm);
        
    \path[->]
    ([xshift=-0.71cm]s3.west) edge +(0.75cm,0);

    \node[state]   (t1) [right of=s2,xshift=1.5cm]  {$1$};
    \node[state]   (t2) [right of=t1] {$2$};
    \node[state]   (t3) [below of=t2,yshift=1cm] {$3$};
    \path[->]
    (t1) edge [loop left] node [left] {$2a$} (t1)
    (t2) edge [loop left] node [left] {$3a$} (t2)
    (t1) edge node [above] {$b$} (t3)
    (t2) edge node [right] {$c$} (t3);

    \path[->] +([yshift=0.75cm]t1.north) edge (t1.north);
    \path[->] +([yshift=0.75cm]t2.north) edge (t2.north);
    \path[->] +(t3.south) edge ([yshift=-0.75cm]t3.south);
  \end{tikzpicture}
  \caption{\emph{Left: \cref{exm:polya}.} A weighted automaton recognizing a Pólya series.
    The automaton is minimal but ambiguous.
    A non-minimal unambiguous automaton on four states recognizes the same series (see \cref{exm:unambig,fig:unambig}).
    \emph{Right: \cref{exm:asymmetric}.} A weighted  automaton that is not determinizable and has a two-dimensional \emph{left} linear hull. The same series can be recognized by a co-deterministic weighted automaton (that is, deterministic when reading words from right to left).
    Correspondingly the \emph{right} linear hull is one-dimensional.} \label{fig:polya}\label{fig:asymmetric}
\end{figure}

\begin{example} \label{exm:asymmetric}
  Let $K= \bQ$, let $X=\{a,b,c\}$, and define a minimal linear representation $(u,\mu,v)$ by $u=(1,1,0)$, $v = (0,0,1)^T$, and
  \begin{align*}
     \mu(a) &= \begin{pmatrix} 2 & 0 & 0 \\ 0 & 3 & 0 \\ 0 & 0 & 0 \end{pmatrix},
    & \mu(b) &= \begin{pmatrix} 0 & 0 & 1 \\ 0 & 0 & 0 \\ 0 & 0 & 0 \end{pmatrix},
    & \mu(c) &= \begin{pmatrix} 0 & 0 & 0 \\ 0 & 0 & 1 \\ 0 & 0 & 0 \end{pmatrix}.
  \end{align*}
  The automaton is depicted in \cref{fig:asymmetric}. Here the (left) linear hull is $\overline{\Omega}=\overline{u \mu(X^*)} = \langle e_1,e_2 \rangle \cup \langle e_3 \rangle$.
  The dually defined right linear hull is $\overline{\mu(X^*) v} = \langle e_1 \rangle \cup \langle e_2 \rangle \cup \langle e_3 \rangle$.
  Observe that neither the dimension nor the number of components of the left and right linear hull coincide.

  Let $S$ be the series recognized by $(u,\mu,v)$. Then $S(a^nb) = 2^n$ and $S(a^nc) = 3^n$ for every $n \ge 0$ and $S(w) = 0$ for every other word $w$.
  By our convention, automata read words left to right.
  In view of \cref{t:determinizable}, and its natural dual for reading words right to left, the asymmetry in the linear hull reflects that $S$ can be recognized by a deterministic automaton when reading words right to left, but not when reading words left to right.
\end{example}

Our next goal is to show that we can change the linear representation in such a way that the irreducible components of the linear hull, which are vector subspaces, form a direct sum (\cref{l:repr-directsum}).
We do this by forming the \emph{external} direct sum of the irreducible components of the linear hull, thereby passing to a linear representation of possibly larger dimension.
For instance, in \cref{exm:polya}, the linear hull is a union of two planes in affine 3-space, necessarily intersecting in a line, while the new linear representation will be defined in affine 4-space and have a linear hull consisting of two 2-dimensional planes (intersecting only in the origin).
We will moreover do this in such a way that, if $S$ is a Pólya series, then all coefficients appearing in vectors of $\Omega$ have their coordinates in $G_0$.

Let $Y \subseteq V$ be a closed subset of a vector space $V$, and $\varphi\in \End_K(V)$ with $\varphi(Y) \subseteq Y$.
Then $\varphi$ is continuous in our topology.
Thus, if $Z$ is an irreducible component of $Y$, there exists an irreducible component $Z'$ of $Y$ such that $\varphi(Z) \subseteq Z'$ (see \ref{dense:continuous-component} of \cref{l:dense}).
In particular, there exists a map $f\colon \irr(Y) \to \irr(Y)$ such that $\varphi(Z) \subseteq f(Z)$ for all $Z \in \irr(Y)$.
In general, there are several possible choices for this map if $\varphi(Z)$ lies in an intersection of multiple irreducible components.

\begin{definition} \label{d:hat}
  Let $V$ be a finite-dimensional vector space and $Y \subseteq V$ a closed set.
  We define a vector space $\widehat Y$ together with a homomorphism $\sigma\colon \widehat Y \to V$ by
  \[
    \widehat Y = \bigoplus_{Z \in \irr(Y)}  Z
  \]
  and $\sigma = \sum_{Z \in \irr(Y)} \pi_Z$, where $\pi_Z \colon \widehat Y \to Z \subseteq V$ denotes the canonical projection.

  For $Z \in \irr(Y)$, let $\varepsilon_Z \colon Z \to \widehat Y$ denote the canonical embedding.
  For $\varphi\in \End_K(V)$ with $\varphi(Y) \subseteq Y$ and a map $f\colon \irr(Y) \to \irr(Y)$ with $\varphi(Z) \subseteq f(Z)$ for all $Z \in \irr(Y)$, let
  \[
    \widehat \varphi = {}_f \widehat \varphi \colon \widehat Y \to \widehat Y \quad\text{be defined by}\quad \widehat \varphi \circ \varepsilon_Z = \varepsilon_{f(Z)} \circ \varphi|_Z.
  \]
\end{definition}

The definition of $\widehat \varphi$ strongly depends on the choice of $f$.
However, as we will see in a moment, $\sigma \circ \widehat \varphi$ does not depend on $f$.
Since we will ultimately be interested in this composition, the particular choice of $f$ will not matter, and we suppress $f$ in the notation when this does not cause confusion.

\begin{lemma} \label{l:hatprop}
  Let all notation be as in \cref{d:hat}.
  \begin{enumerate} 
  \item \label{hat:value} For $Z \in \irr(Y)$ and $z \in Z$, we have $\varphi(z) = \sigma \circ \widehat \varphi \circ \varepsilon_Z(z)$.
  \item \label{hat:composition} If $\varphi$, $\psi \in \End_K(V)$ with $\varphi(Y) \subseteq Y$ and $\psi(Y) \subseteq Y$, then
    \[
      \sigma \circ {}_h\widehat{(\psi \circ \varphi)} = \sigma \circ {}_g \widehat \psi \circ {}_f \widehat \phi.
    \]
    for any choice of $f$, $g$,~$h \colon \irr(Y) \to \irr(Y)$ with $\varphi(Z) \subseteq f(Z)$, $\psi(Z) \subseteq g(Z)$, and $\psi \circ \varphi(Z) \subseteq h(Z)$ for all $Z \in \irr(Y)$.
  \end{enumerate}
\end{lemma}

\begin{proof}
  \ref*{hat:value}
  We have
  \[
    \begin{split}
      \sigma \circ \widehat \varphi \circ \varepsilon_Z(z) &= \sigma \circ \varepsilon_{f(Z)} \circ \varphi(z) = \sum_{Z' \in \irr(Y)} \pi_{Z'} \circ \varepsilon_{f(Z)} \circ \varphi(z)\\
      &= \pi_{f(Z)} \circ \varepsilon_{f(Z)} \circ \varphi(z) = \varphi(z).
    \end{split}
  \]

  \ref*{hat:composition}
  Let $Z \in \irr(Y)$ and $z \in Z$.
  Then, by applying \ref*{hat:value}, we have
  \[
    \sigma \circ {}_h\widehat{(\psi \circ \varphi)} \circ \varepsilon_Z(z) = \psi \circ \varphi(z).
  \]
  On the other hand,
  \[
    \sigma \circ {}_g \widehat \psi \circ {}_f \widehat \varphi \circ \varepsilon_Z(z) = \sigma \circ {}_g \widehat \psi \circ \varepsilon_{f(Z)} \circ \varphi(z) = \psi(\varphi(z)),
  \]
   where have again applied \ref*{hat:value} in the second equality.
\end{proof}

\begin{lemma} \label{l:good-basis}
  Let $S \in K\llangle X \rrangle$ be a rational series and $\Gamma = \{\, S(w) : w \in X^* \,\} \subseteq K$.
  Then $S$ has a minimal linear representation $(u, \mu, v)$ with $u\mu(w) \in \Gamma^{1\times n}$ for all $w \in X^*$.
  If $n\ge 1$, we can take $v=e_1$.
\end{lemma}

\begin{proof}
  Let $(u',\mu',v')$ be a minimal linear representation of $S$.
  By minimality we have $\langle \mu'(w)v' : w \in X^* \rangle_K = K^{n \times 1}$.
  Let $w_1$, $\ldots\,$,~$w_n \in X^*$ be such that the vectors $b_i = \mu'(w_i)v'$ for $i \in [1,n]$ form a basis of $K^{n \times 1}$.
  If $n\ge 1$, then $v' \ne 0$ by minimality, and we may take $w_1=1$ (the empty word), ensuring $b_1 = v'$.
  Let $B \in K^{n\times n}$ be the matrix whose $i$-th column is $b_i$, and let $\mu \colon X^* \to K^{n\times n}$ be defined by $\mu(w) = B^{-1}\mu'(w)B$.
  Let $u=u'B$ and $v=B^{-1}v'$.
  Then $(u,\mu,v)$ is a minimal linear representation of $S$.

  Let $w \in X^*$ and $i \in [1,n]$.
  Then $u \mu(w) = u'\mu'(w)B$.
  The $i$-th coordinate of this vector is
  \[
    u'\mu'(w) b_i = u'\mu'(ww_i)v' = S(ww_i) \in \Gamma.
  \]
  By construction, we also have $v=B^{-1}v'=e_1$ if $n \ge 1$.
\end{proof}

\begin{lemma} \label{l:density-sub}
  Let $V$ be a vector space with basis $e_1$, $\ldots\,$,~$e_n$.
  Let $\Gamma \subseteq K$.
  For every vector subspace $W \subseteq V$, there exists a basis $f_1$, $\ldots\,$,~$f_m$ of $W$ such that
  \[
    (\Gamma e_1 + \dots + \Gamma e_n) \cap W \subseteq \Gamma f_1 + \cdots + \Gamma f_m.
  \]
\end{lemma}

\begin{proof}
  Using standard reductions and possibly renumbering the basis elements of $V$, we can find a basis $f_1$, $\ldots\,$,~$f_m$ of $W$ with $e_i^*(f_j) = \delta_{i,j}$ for $i$, $j \in [1,m]$ (essentially Gram-Schmidt).

  Let
  \[
    w=\alpha_1 e_1 + \dots + \alpha_n e_n = \beta_1 f_1 + \dots + \beta_m f_m
  \]
  with $\alpha_1$, $\ldots\,$,~$\alpha_n \in \Gamma$ and $\beta_1$, $\ldots\,$,~$\beta_m \in K$.
  Applying $e_i^*$ to the equation shows $\beta_i = \alpha_i \in \Gamma$ for $i \in [1,m]$.
\end{proof}

\begin{lemma} \label{l:repr-directsum}
  Let $(u,\mu,v)$ be a linear representation representing a rational series $S \in K\llangle X \rrangle$ and $\Omega = u \mu(X^*)$.
  Let $\Gamma \subseteq K$ be such that $\Omega \subseteq \Gamma^{1\times n}$ and $\Omega v \subseteq \Gamma$.
  Let $\irr(\overline \Omega) = \{ W_1, \ldots, W_k \}$, let $m_i= \dim W_i$, and let $m=m_1+\dots+m_k$.
  
  For $i \in [1,k]$, let $W_i' \subseteq K^{1 \times m}$ be the $m_i$-dimensional subspace spanned by the standard basis vectors $e_{j} \in K^{1\times m}$ with $j \in [m_1 + \cdots + m_{i-1}+1, m_1+\cdots+m_i]$, so that in particular
  \[
    K^{1 \times m} = W_1' \oplus \cdots \oplus W_k'.
  \]
  Then there exists an $m$-dimensional linear representation $(u',\mu',v')$ of $S$ with
  \[
    \irr(\overline \Omega') = \{ W_1', \dots, W_k' \}
  \]
  where $\Omega' = u' \mu'(X^*)$.
  Moreover $\Omega' \subseteq \Gamma^{1 \times m}$ and $\Omega'v' \subseteq \Gamma$.
\end{lemma}

\begin{proof}
  If $n=0$, then $(u',\mu',v')=(u,\mu,v)$ trivially has the desired properties.
  To avoid this degenerate case, from now on assume $n \ge 1$.

  For each $W_i$ choose a basis $f_{(i,1)}$, $\ldots\,$,~$f_{(i,m_i)}$ as in \cref{l:density-sub}.
  Denote by $e_1$, $\ldots\,$,~$e_n$ the standard basis of $K^{1\times n}$.
  For $x \in X$ let $\varphi_x \colon K^{1\times n} \to K^{1 \times n}$ denote the homomorphism that is represented by $\mu(x)$, and let $\psi\colon K^{1\times n} \to K$ denote the homomorphism represented by $v$.
  That is, $u \mu(x_{1} \cdots x_{l}) v = \psi \circ \varphi_{x_l} \circ \dots \circ \varphi_{x_1}(u)$ for $x_1$, $\ldots\,$,~$x_l \in X$.

  For each $x \in X$, the homomorphism $\varphi_x$ is continuous and closed and $\varphi_x(\Omega) \subseteq \Omega$.
  Hence also $\varphi_x(\overline \Omega) \subseteq \overline \Omega$.
  Without restriction $u \in W_1$.
  Let $Y=\overline \Omega$.
  We denote by $\varepsilon_i \colon W_i \to \widehat Y$ the canonical embedding.
  By \cref{l:hatprop},
  \begin{equation}  \label{e:expanded-rep}
    \psi \circ \varphi_{x_l} \circ \dots \circ \varphi_{x_1}(u) = \psi \circ \sigma \circ \widehat \varphi_{x_l} \circ \dots \circ \widehat \varphi_{x_1} \circ \varepsilon_1(u).
  \end{equation}

  Set $Q = \{\, (i,j) : i \in [1,k],\, j \in [1,m_i] \,\}$.
  Then the family $(\varepsilon_i(f_{(i,j)}))_{(i,j)\in Q}$ is a basis of $\widehat Y$.
  With respect to this choice of basis, let $u' \in K^{1\times Q}$ represent $\varepsilon_1(u)$, let $v' \in K^{Q \times 1}$ represent $\psi \circ \sigma \colon K^{1 \times Q} \to K$, and let $A'_x$ represent $\widehat \varphi_x \colon K^{Q \times Q} \to K^{Q \times Q}$.
  Setting $\mu'(x)=A'_x$, the tuple $(u', \mu', v')$ is a linear representation of $S$ by \cref{e:expanded-rep}.

  Clearly $\Omega' \subseteq \varepsilon_1(W_1) \cup \cdots \cup \varepsilon_k(W_k)$, and we show that the right side is the decomposition of $\overline{\Omega'}$ into irreducible components.
  The sets $\varepsilon_i(W_i)$ are irreducible closed subsets of $\widehat Y$, and $\varepsilon_i(W_i) \cap \varepsilon_j(W_j) = 0$ if $i \ne j$.
  Let $i \in [1,k]$.
  The set
  \[
    \Omega_i = (\Omega \cap W_i) \setminus \bigcup \{\, W_j : j \in [1,k],\, i \ne j \,\}
  \]
  is dense in $W_i$.
  If $w \in X^*$ with $u \mu(w) \in \Omega_i$, then necessarily $u'\mu'(w) \in \varepsilon_i(W_i) \cap \Omega'$.
  Hence $\varepsilon_i(W_i) \cap \Omega'$ is dense in $\varepsilon_i(W_i)$.

  Finally, by choice of the $f_{(i,j)}$, we have
  \[
    \Omega \cap W_i \subseteq (\Gamma e_1 + \cdots + \Gamma e_n) \cap W_i \subseteq \Gamma f_{(i,1)} + \cdots + \Gamma f_{(i,m_i)}.
  \]
  Thus $u'\mu'(w) \in \Gamma^{1 \times Q}$ for all $w \in X^*$.
  If $w \in X^*$ with $u'\mu'(w) \in \Omega'$, then $u\mu(w) \in \Omega$ and $u'\mu'(w)v' = u\mu(w)v \in \Gamma$.
\end{proof}

We are now in a position to deal with the (relatively easy) case in which the linear hull has dimension at most $1$.
As an immediate corollary we obtain the direction \ref{main:polya}$\,\Rightarrow\,$\ref{main:unam-auto} of \cref{t:main} in the special case where $S$ has only \emph{finitely many} distinct coefficients.

\begin{proposition} \label{p:deterministic}
  Let $S \in K\llangle X \rrangle$ be a rational series with a linear representation $(u,\mu,v)$ whose linear hull has dimension $\le 1$.
  Then $S$ is recognized by a deterministic weighted automaton.
\end{proposition}

\begin{proof}
  Replacing the representation by one as in \cref{l:repr-directsum}, we can assume that the spaces in $\irr(\overline \Omega)$ form a direct sum.
  Thus, if $W \in \irr(\overline \Omega)$, $a \in W$, and $x \in X$ with $0 \ne a \mu(x)$, then there exists a unique $W' \in \irr(\overline \Omega)$ with $a \mu(x) \in W'$.
  This means that each column of $\mu(x)$ contains at most one non-zero entry.
  Hence the weighted automaton associated to this linear representation is deterministic.
\end{proof}

\begin{proof}[Proof of \cref{c:finite}]
  Choose a representation as in \cref{l:good-basis} with $\Gamma = \{\, S(w): w \in X^* \, \}$ being finite.
  Then $\Omega = u \mu(X^*) \subseteq \Gamma^{1 \times n}$ is a finite set, and hence $\overline \Omega$ is a finite union of vector spaces of dimension $\le 1$.
  Thus $\dim \overline \Omega \le 1$, and \cref{p:deterministic} implies the claim.
\end{proof}

\begin{remark}
  \Cref{c:finite} was known before.
  If $G$ is finite, then $S$ is a finite linear combination of characteristic series of rational languages \cite[Corollary 3.2.6]{berstel-reutenauer11}.
  Since each of these characteristic series can be recognized by a deterministic automaton, the result follows \cite[\S 6]{reutenauer96}.
  However, the proof we give here is more in line with the one for our general result.
\end{remark}

\section{An important lemma}
\label{sec:unit-equations}

In this section we again consider only the case where $R=K$ is a field.
We are now ready to prove a key lemma in characteristic $0$.
Its proof depends on unit equations.
As a consequence, a variant for positive characteristic is more complicated and will follow at the end of the section.

We recall the fundamental finiteness result on unit equations in characteristic $0$ that we will be using.
For number fields it was proved independently by Evertse \cite{evertse84} and van der Poorten--Schlickewei \cite{vdpoorten-schlickewei82}; the extension to arbitrary fields appears in \cite{vdpoorten-schlickewei91}.
We refer to \cite[Chapter 6]{evertse-gyory15} or \cite[Theorem 7.4.1]{bombieri-gubler06} for more details.

\begin{proposition}[Evertse; van der Poorten--Schlickewei]
  Suppose $\chr K = 0$.
  Let $m \ge 2$, and $a_1$, $\ldots\,$,~$a_m \in K^\times$.
  Then there exist only finitely many projective points $(x_1: \cdots : x_m)$ with coordinates $x_1$, $\ldots\,$,~$x_m \in G$ such that
  \begin{equation} \label{eq:sunit}
    a_1 x_1 + \cdots + a_m x_m = 0
  \end{equation}
  and $\sum_{i \in I} a_i x_i \ne 0$ for any non-empty, proper subset $I$ of $[1,m]$.
\end{proposition}

A solution $(x_1,\ldots,x_m)$ of \eqref{eq:sunit} with $\sum_{i\in I} a_i x_i \ne 0$ for every $\emptyset \ne I \subsetneq [1,m]$ is called \defi{non-degenerate}.
So, counted as projective points, there are only finitely many non-degenerate solutions with coordinates in $G$.
It is easily seen that there can be infinitely many degenerate solutions (even when considered as projective points), but by definition, the affine coordinates of the degenerate solutions lie in a finite union of proper vector subspaces.

\begin{lemma} \label{l:functional}
  Suppose $\chr K=0$.
  Let $V$ be a vector space with basis $e_1$, $\ldots\,$,~$e_n$.
  Suppose that $\Omega \subseteq G_0e_1 + \dots + G_0e_n$ is a dense subset of $V$.
  Then, for all $\varphi\in \Hom_K(V,K)$ with $\varphi(\Omega) \subseteq G_0$, there exists at most one $i \in [1,n]$ with $\varphi(e_i) \ne 0$.
\end{lemma}

\begin{proof}
  The claim is trivial for $n=1$.
  Suppose $n \ge 2$.
  With respect to the basis $(e_1,\ldots,e_n)$, the homomorphism $\varphi$ is represented by $(\alpha_1, \ldots, \alpha_n) \in K^{1\times n}$ with $\alpha_i=\varphi(e_i)$.
  Enlarging $G$ if necessary, we may assume $\alpha_1$, $\ldots\,$,~$\alpha_n \in G_0$.

  Let $I = \{\, i \in [1,n] : \alpha_i \ne 0 \,\}$.
  We must show $\card{I} \le 1$.
  Suppose to the contrary that $\card{I} \ge 2$.
  After renumbering the basis vectors if necessary, we may assume $I = [1,m]$ for some $m \ge 2$.
 
  For $\emptyset \neq J \subseteq I$ let
  \[
    V_J = \big\{\, \lambda_1 e_1 + \cdots + \lambda_n e_n \in V :  \sum_{j \in J} \alpha_j \lambda_j = 0 \,\big\}.
  \]
  By choice of $I$, each $V_J$ is a proper subspace of $V$.
  Set $Y = \bigcup_{\emptyset \neq J \subseteq I} V_J$.
  Since $K$ is an infinite field, a vector space cannot be covered by a finite union of proper subspaces.
  Thus $Y \subsetneq V$.
  Hence $\Omega' = \Omega \setminus Y$ is dense in $V$ by \cref{l:dense}.

  If $v = \lambda_1 e_1 + \cdots + \lambda_n e_n \in \Omega'$, then
  \begin{equation} \label{eq:seq}
    \sum_{i \in I} \alpha_i \lambda_i = g,
  \end{equation}
  for some $g$ in $G_0$ by assumption on $\varphi$.
  If $\emptyset \ne J \subseteq I$, then
  \[
    \sum_{j \in J} \alpha_j \lambda_j \ne 0,
  \]
  since $v \not \in V_J$.
  Thus $(\alpha_1 \lambda_1, \dots, \alpha_m \lambda_m, -g)$ is a non-degenerate solution of the unit equation $X_1 + \cdots + X_{m+1} = 0$.

  Hence there exists a finite subset $M \subseteq \mathbb P^{m-1}(K)$ with $(\alpha_1 \lambda_1 \colon \cdots \colon \alpha_m \lambda_m) \in M$ for all $v=\lambda_1 e_1 + \cdots + \lambda_n e_n \in \Omega'$.
  In particular, since $m \ge 2$, we see that $\lambda_1/\lambda_2$ can take only finitely many values.
  Thus $\Omega'$ can be covered by finitely many proper vector subspaces of $V$, in contradiction to $\overline{ \Omega'}=V$.
\end{proof}

\subsection{Positive characteristic}

For this subsection, we now make the additional assumptions that $\chr K = p > 0$, and that $K$ is finitely generated over its prime field $\bF_p$.

In extending \cref{l:functional} to positive characteristic, we face the problem that unit equations may have infinitely many non-degenerate solutions.
However, a result of Derksen and Masser \cite{derksen-masser12} is useful in bounding the number of solutions of bounded height.
In this way, we will be able to recover a version of \cref{l:functional} with the original density hypothesis replaced by a quantitative one.

As in \cite[Section 2]{derksen-masser12}, we can define a set of discrete valuations and associated absolute values on $K$ in such a way that the absolute values satisfy the product formula.
These absolute values depend on a choice of transcendence basis; we will always work with a fixed such set.

Associated to this set of absolute values, we define a (logarithmic, projective) height of an element $a = (\alpha_1 : \cdots : \alpha_{n}) \in \bP^{n-1}(K)$ by
\[
  h(a) = \log \prod_{v} \max\{ \abs{\alpha_1}_v, \ldots, \abs{\alpha_n}_v \}.
\]
For $0 \ne (\alpha_1,\ldots,\alpha_n) \in K^{1 \times n}$ we set $h(\alpha_1, \ldots, \alpha_n) = h(\alpha_1 : \cdots : \alpha_n)$.
This height satisfies the Northcott property, that is, for every $N \ge 0$, the set $\{\, a \in \bP^{n-1}(K) : h(a) \le N \,\}$ is finite (at this point we are using that the prime field is finite).
Moreover, for every $A \in K^{n \times n}$ there exists a constant $C_A$, such that for any $a \in K^{1 \times n}$ with $a \not \in \ker A$,
\[
  h(aA) \le h(a) + C_A.
\]
If $A$ is invertible, then even
\[
  h(a) - C_A \le h(aA) \le h(a) + C_A.
\]
(See \cite[Theorem B.2.5]{hindry-silverman00}.)

If $V$ is a finite-dimensional vector space, then any choice of basis gives an isomorphism $V \to K^{1\times n}$, and therefore induces a corresponding height $h_V$ on $V$ and on the projective space $\bP(V)$.
If $h_V$, $h_V'$ are two such heights, induced by different bases, then $h_V'(a) = h_V(a) + O(1)$.
The exact choice of height will not matter.

We shall also need the following property.
\begin{lemma} \label{l:fingenroot}
  The group
  \[
    \sqrt{G} \coloneqq \{\, a \in K : a^n \in G \text{ for some $n \ge 1$}\,\} \le K^\times
  \]
  is finitely generated.
\end{lemma}

\begin{proof}
  By assumption $K$ is finitely generated over its prime field $\bF_p$ and $G \le K^\times$ is a finitely generated subgroup.
  Let $\bF_q$ be the algebraic closure of $\bF_p$ in $K$, so that $K$ is a regular extension of $\bF_q$.
  Let $R$ be the finitely generated $\bF_q$-subalgebra of $K$ generated by $G$.
  Then the integral closure $\overline R$ is a finitely generated $R$-module by \cite[Proposition 2.4.1]{lang83} or \cite[Corollary 13.13]{eisenbud95}.
  Hence $\overline R$ also is a finitely generated $\bF_q$-algebra.
  Now \cite[Corollary 2.7.3]{lang83} implies that $\overline{R}^\times$ is a finitely generated group; since $\sqrt{G} \subseteq \overline{R}^\times$, the claim follows.
\end{proof}

Let
\[
  \bP^{n-1}(G) = \{\, (\alpha_1: \cdots : \alpha_n) : \alpha_1, \ldots, \alpha_n \in G \,\}.
\]
As a consequence of a theorem of Derksen--Masser \cite[Theorem 3]{derksen-masser12}, we have an upper bound on the number of solutions of bounded height of a unit equation in positive characteristic.

\begin{lemma} \label{l:derksen-masser-cor}
  Let $n \ge 2$ and let $S \subseteq \bP^{n-1}(G)$ be the set of non-degenerate solutions in $G$ to
  \begin{equation} \label{eq:p-unit-equation}
    a_1 x_1 + \cdots + a_n x_n = 0, \qquad{a_1, \ldots, a_n \in K^\times},
  \end{equation}
  and for all $N \ge 0$ let $c(N) \coloneqq \card{ \{\, a \in S : h(a) \le N \,\} }$.
  Then there exists $D \in \bZ_{\ge 0}$ such that
  \begin{equation} \label{eq:p-upper-bound}
    c(N) = O(\log(N)^D).
  \end{equation}
\end{lemma}

\begin{proof}
  By the \cref{l:fingenroot}, the group $\sqrt{G}$ is finitely generated.
  A $\sqrt G$-automorphism is a map
  \[
    \psi \colon \bP^{n-1}(K) \to \bP^{n-1}(K),\ (\alpha_1:\cdots:\alpha_n) \mapsto (g_1 \alpha_1: \ldots : g_n \alpha_n)
  \]
  with $g_1$, $\ldots\,$,~$g_n \in \sqrt{G}$.
  For $q$ a power of $p$, let
  \[
    \varphi_q(\alpha_1:\cdots:\alpha_n) = (\alpha_1^q: \cdots : \alpha_n^q).
  \]
  Finally, for $\sqrt{G}$-automorphisms $\psi_1$, $\ldots\,$,~$\psi_k$ and $a \in \bP^{n-1}(K)$ let
  \[
    [ \psi_{1}, \ldots, \psi_{k} ]_q(a) \coloneqq  \{\, (\psi_{1}^{-1}\varphi_q^{e_1} \psi_{1}) (\psi_{2}^{-1} \varphi_q^{e_2} \psi_{2}) \cdots (\psi_{k}^{-1} \varphi_q^{e_k} \psi_{k})(a) : e_1, \ldots, e_k \in \bZ_{\ge 0} \,\}.
  \]
  (We suppress the composition operator $\circ$ for brevity.)

  By \cite[Theorem 3]{derksen-masser12}, the set of solutions $S$ of \eqref{eq:p-unit-equation} is contained in a finite union of sets of the form $[ \psi_{1}, \ldots, \psi_{k} ]_q(a)$.
  It therefore suffices to show that \eqref{eq:p-upper-bound} holds for $S$ such a set.
  Thus, suppose $S=[ \psi_{1}, \ldots, \psi_{k} ]_q(a)$ for some $a \in \bP^{n-1}(K)$ and $\sqrt{G}$-automorphisms $\psi_1$, $\ldots\,$,~$\psi_k$.

  By induction on $k$, we show that $S$ contains $O((\log N)^k)$ elements of height at most $N$.
  The case $k=0$ is clear.
  Suppose $k \ge 1$ and that the claim holds for $k-1$.
  There exists $C' \ge 0$ such that $h(\psi_1^{-1}(b)) \ge  h(b) - C'$ and $h(\psi_1(b)) \ge h(b) - C'$ for all $b \in \bP^{n-1}(K)$.
  Let
  \begin{align*}
    T_0 &= \{\, b \in [ \psi_{2}, \ldots, \psi_{k} ]_q(a) : h(b) \le C'+1 \,\}, \text{ and}\\
    T_1 &= \{\, b \in [ \psi_{2}, \ldots, \psi_{k} ]_q(a) : h(b) > C'+1 \,\}.
  \end{align*}

  The set $T_0$ is finite.
  For each $b \in T_0$ and $b'=\psi_1^{-1} \varphi_q^{e_1} \psi_1(b)$ with $e_1 \ge 0$, we have $h(b') \ge q^{e_1} h(\psi_1(b)) - C'$.
  Hence, the number of such elements $b'$ with $h(b') \le N$ is $O_b(\log N)$.
  Exploiting the finiteness of $T_0$, altogether there are $O(\log N)$ elements $b' \in \psi_1^{-1}\varphi_q^{e_1}\psi_1(T_0)$, with $e_1 \ge 0$, such that $h(b') \le N$.

  For $b \in T_1$ and $b'=\psi_1^{-1} \varphi_q^{e_1} \psi_1(b)$ with $e_1 \ge 0$, we have
  \[
    h( b' ) \ge q^{e_1}( h(b) - C') - C' > q^{e_1} - C'.
  \]
  Thus, if $h(b') \le N$ then $e_1 \le \log(N+C')/\log(q)$ and $h(b) \le N + 2C'$.
  By the induction hypothesis, there are $O(\log(N+2C')^{k-1}) = O(\log(N)^{k-1})$ elements $b \in T_1$ with $h(b) \le N + 2C'$.
  Thus there are $O(\log(N)^k)$ elements $b' \in \psi_1^{-1} \varphi_q^{e_1} \psi_1(T_1)$, with $e_1 \ge 0$, for which $h(b') \le N$.

  Altogether,
  \[
    S = \bigcup_{e_1 \ge 0} \psi_1^{-1} \varphi_q^{e_1} \psi_1(T_0) \,\cup\, \psi_1^{-1} \varphi_q^{e_1} \psi_1(T_1)
  \]
  contains $O(\log(N)^k)$ elements of height at most $N$.
\end{proof}

We can now obtain a variant of \cref{l:functional}, using a slightly stronger hypothesis, that also holds in positive characteristic.
To ensure that a vector subspace of $K^{1 \times n}$ is irreducible in our topology, we do still need to assume that $K$ is infinite.

For a subset $S$ of a vector space $V$, let $\bP(S)$ be the image of $S\setminus \{0\}$ in the projective space $\bP(V)$.

\begin{lemma} \label{l:functional-p}
  Suppose $K$ is an \emph{infinite} field of positive characteristic, finitely generated over its prime field.
  Let $V$ be a vector space with basis $e_1$, $\ldots\,$,~$e_n$, and $\Omega \subseteq G_0e_1 + \dots + G_0e_n$ a dense subset of $V$.
  Assume that, for every standard projection $\pi\colon V \to W$ with $W=\langle e_{i_1}, \ldots, e_{i_m} \rangle_K$ and $m \ge 2$, every closed subset $Y \subsetneq W$, and every $C \in \bR_{\ge 0}$, $D \in \bZ_{\ge 0}$, there exist arbitrarily large $N$ such that
  \[
    \card{\{\, a \in \bP(\pi(\Omega) \setminus Y) : h_W(a) \le N \,\}} > C \log(N)^D.
  \]

  Then, for all $\varphi\in \Hom_K(V,K)$ with $\varphi(\Omega) \subseteq G_0$, there exists at most one $i \in [1,n]$ with $\varphi(e_i) \ne 0$.
\end{lemma}

\begin{proof}
  The beginning of the proof and the overall strategy are analogous to \cref{l:functional}.
  Let $\varphi(e_i) = \alpha_i \in G_0$ and $I = \{\, i \in [1,m] : \alpha_i \ne 0 \,\}$.
  Without restriction $I = [1,m]$ and we have to show $m=1$.
  Assume $m \ge 2$.
  Let $\pi \colon V \to W$ with $W=\langle e_1,\ldots,e_m \rangle_K$ denote the standard projection.

  For $\emptyset \ne J \subseteq I$, let $W_J = \{\, \lambda_1e_1 + \cdots +  \lambda_m e_m \in W : \sum_{j \in J} \alpha_j \lambda_j = 0\,\}$ and $Y = \bigcup_{\emptyset \ne J \subseteq I} W_J$.
  Note that $Y \subsetneq W$.
  Thus, applying our assumption to points of bounded height in $\bP(\pi(\Omega) \setminus Y)$, we find that, for any $D \in \bZ_{\ge 0}$, the set $\bP(\pi(\Omega) \setminus Y)$ contains more than $O(\log(N)^D)$ points of height at most $N$.

  Suppose $\lambda_1 e_1 + \cdots + \lambda_m e_m$ represents a point in $\bP(\pi(\Omega) \setminus Y)$.
  Then there exists a $g \in G$ with
  \[
    \sum_{j=1}^m \alpha_i \lambda_i = g.
  \]
  Hence $(\alpha_1 \lambda_1 : \cdots : \alpha_m \lambda_m : -g)$ is a non-degenerate solution of the unit equation $X_1 + \cdots + X_{m+1} = 0$.
  By \cref{l:derksen-masser-cor}, we conclude that there exist $O(\log(N)^D)$ such points $(\alpha_1 \lambda_1 : \cdots : \alpha_m \lambda_m)$ of height at most $N$, a contradiction to the size of $\bP(\pi(\Omega) \setminus Y)$.
\end{proof}

The following example shows that the conclusion of the previous lemma is trivially false for finite fields.

\begin{example}
  Let $K$ be a finite field and $n \ge 2$.
  Let $A_1$, $\ldots\,$,~$A_k \in \GL(n,K)$ be such that the residue classes generate $\PGL(n,K)$ as a semigroup, and let $X=\{a_1,\ldots,a_k\}$.
  Let $0 \ne u \in K^{1 \times n}$, let $0 \ne v \in K^{n \times 1}$, and let $\mu(a_i)=A_i$ for $i \in [1,k]$.
  Since $\PGL(n,K)=\mu(X^*)$ acts transitively on $\bP^{n-1}(K)$, the linear representation $(u,\mu,v)$ is minimal and its linear hull is $K^{1\times n}$, that is, the set $\Omega=u\mu(X^*)$ is dense in all of $K^{1 \times n}$.
  With $G =K^\times$ and $G_0 =K$, any choice of $\lambda_1$, $\ldots\,$,~$\lambda_n \in K$, yields a linear map $\varphi\colon K^{1\times n} \to K, (\beta_1,\ldots,\beta_n) \mapsto \lambda_1 \beta_1 + \cdots + \lambda_n \beta_n$ with $\varphi(\Omega) \subseteq G_0$.
  In particular, we may take all $\lambda_i$ to be nonzero, in contrast to the conclusion of the previous lemma.
\end{example}

We will later apply \cref{l:functional}, respectively \cref{l:functional-p}, to an irreducible component $V$ of the linear hull $\overline{\Omega}=\overline{u\mu(X^*)}$ of a linear representation $(u,\mu,v)$.
In characteristic $0$ we may use \cref{l:functional}, and in this case it is clear that $\Omega \cap V$ is dense in $V$.
However, in positive characteristic, where we need to apply \cref{l:functional-p}, it is necessary to verify that the stronger conditions of this lemma are indeed satisfied.
This is the subject of the next two lemmas.

Let $(u,\mu,v)$ be a linear representation and $\Omega = u \mu(X^*)$.
For $N \in \bZ_{\ge 0}$, let $\Omega_{\le N} = \{\, u\mu(w) : w \in X^*,\, \length{w} \le N \,\}$.

\begin{lemma} \label{l:closed-bound}
  Let $Y \subseteq K^{1\times n}$ be a closed set of dimension $m$ with $k$ irreducible components, and let $N = k^{\card{X}(m-1)}+1$.
  If $\Omega_{\le N} \subseteq Y$ then $\Omega \subseteq Y$.
\end{lemma}

\begin{proof}
  Starting from $Y$ we iteratively construct a sequence of subsequently smaller closed subsets; it is easiest to keep track of the necessary data using disjoint unions of $k$ rooted trees.

  We thus construct a sequence of graphs $\cT_1$, $\ldots\,$, $\cT_l$ whose vertices are labeled by vector spaces contained in $Y$, and with each $\cT_i$ having the following properties:
  \begin{enumerate}
  \item\label{graph:sub} If $W'$ labels a child of a vertex labeled by $W$, then $W' \subsetneq W$.
  \item\label{graph:count} $\cT_i$ is $s$-regular (except for the leaves) with $s=k^{\card{X}}$.
  \item\label{graph:map} If $W$ labels a leaf and $x \in X$, then there exists a vertex labeled by $W'$ such that $W \mu(x) \subseteq W'$.
  \item\label{graph:union} If $W$ labels an internal vertex and $W_1$, $\ldots\,$,~$W_s$ label its children, then
    \[
      \Omega_{\le N-i} \cap W \,\subseteq\, W_1 \cup \cdots \cup W_s.
    \]
  \end{enumerate}

  The graph $\cT_1$ has $k$ roots labeled by the elements of $\irr(Y)$.
  For $W \in \irr(Y)$ and $F\colon X \to \irr(Y)$, let
  \[
    W_F = \{\, a \in W : a \mu(x) \in F(x) \text{ for all $x \in X$} \,\}.
  \]
  If, for a fixed $W$, each of the vector spaces $W_F$ is a proper subspace of $W$, then we attach $s$ children to the root labeled by $W$.
  These children are labeled by $W_F$ for $F\colon X \to \irr(Y)$.
  On the other hand, if $W=W_F$ for some $F$, we do not attach any children to the vertex labeled by $W$.
  In this case, observe that $W\mu(x) \subseteq F(x)$ for every $x \in X$.

  It is clear that $\cT_1$ satisfies \ref{graph:sub} and \ref{graph:count}.
  Property \ref{graph:map} holds by the choice of the spaces.
  For \ref{graph:union} let $a \in \Omega_{\le N-1} \cap W$.
  Taking any $x \in X$, we have $a\mu(x) \in W'$ for some $W' \in \irr(Y)$.
  Then $a \in W_F$ for any $F \colon X \to \irr(Y)$ with $F(x) = W'$.

  We now iteratively construct $\cT_i$ from $\cT_{i-1}$ for $i \ge 2$.
  If, for every leaf of $\cT_{i-1}$, say, labeled by $W$,  and every $x \in X$, there exists a leaf labeled by $W'$ such that $W \mu(x) \subseteq W'$, then we stop and set $l = i-1$.

  Otherwise, fix a leaf $\alpha$ labeled by $W$ and an $x \in X$ such that $W \mu(x)$ is not contained in any label of a leaf of $\cT_{i-1}$.
  By \ref{graph:map} there exists an internal vertex $\beta$ labeled by $W'$ such that $W \mu(x) \subseteq W'$ but $W\mu(x) \not\subseteq W''$ for any $W''$ labeling a child of $\beta$.
  By our construction $W'$ has $s$ children labeled by $W'_1$, $\ldots\,$,~$W'_s$.
  Set $W_j = \{\, a \in W : a \mu(x) \in W'_j \,\}$, and attach $s$ new children to $\alpha$, labeled by $W_1$, $\ldots\,$,~$W_s$.

  It is clear that \ref{graph:sub}--\ref{graph:map} are preserved for $\cT_i$.
  Property \ref{graph:union} also carries over for vertices other than $\alpha$.
  To verify \ref{graph:union} for $\alpha$, let $a \in \Omega_{\le N-i} \cap W$.
  Then $a \mu(x) \in W' \cap \Omega_{\le N-(i-1)}$ and hence $a\mu(x) \in W_j'$ for some $j$ by \ref{graph:union}.
  Thus $a \in W_j$.

  To see that this process terminates, note that each $\cT_i$ is $s$-regular (except for the leaves) of height at most $m$, and hence has at most $s^{m}$ vertices.
  Since each $\cT_i$ has $s$ vertices more than $\cT_{i-1}$, the process terminates after at most $s^{m-1}$ steps, that is $l \le s^{m-1}+1$.

  Instead of \ref{graph:map}, the final graph $\cT_l$ has the stronger property that if $W$ labels a leaf and $x \in X$, then there exists $W'$ labeling a leaf of $\cT_l$ such that $W \mu(x) \subseteq W'$.
  Defining $Y'$ to be the union of all labels of leaves of $T_l$, this implies $Y'\mu(x) \subseteq Y'$ for each $x \in X$ and hence $Y' \mu(X^*) \subseteq Y'$.
  As $Y'$ contains $u \in \Omega_{\le N - s^{m-1} -1}$, this implies $\Omega \subseteq Y' \subseteq Y$.
\end{proof}

 We give an example in dimension $n=3$ with $\cT_2 \ne \cT_1$, illustrating the iterative construction in the previous proof.
\begin{example}
  Let $e_1$,~$e_2$,~$e_3 \in K^{1\times 3}$ be the standard unit vectors and let $X=\{a\}$.
  Consider $(u,\mu,v)$ with $u=e_1$, with $v=e_1^T$, and with $\mu(a) \in K^{3\times 3}$ the permutation matrix defined by $e_1\mu(a)=e_2$, $e_2\mu(a)=e_3$, and $e_3\mu(a)=e_1$. Clearly $\Omega = \{e_1,e_2,e_3\}$ and $\overline \Omega = \langle e_1 \rangle \cup \langle e_2 \rangle \cup \langle e_3 \rangle$.
  Let $Y=\langle e_1,e_2 \rangle \cup \langle e_2,e_3 \rangle$ and note $\Omega \subseteq Y$.

  Now $\cT_1$ has two roots, labeled by $W_{1,2} \coloneqq \langle e_1,e_2 \rangle$ and $W_{2,3} \coloneqq \langle e_2,e_3\rangle$.
  Since $W_{1,2}\mu(a) = W_{2,3}$, no children are attached to $W_{1,2}$ in $\cT_1$.
  However $W_{2,3}\mu(a) = \langle e_1,e_3\rangle \not\subseteq Y$.
  Thus, in $\cT_1$, two children are attached to the root labeled by $W_{2,3}$.
  These children are labeled by $\langle e_2 \rangle = \{\,w \in W_{2,3} : w\mu(a) \subseteq W_{2,3} \,\}$ and $\langle e_3 \rangle = \{\, w \in W_{2,3} : w\mu(a) \in W_{1,2} \,\}$.

  In the next step, we observe that $W_{1,2}$ is not mapped into labels of leaves of $\cT_1$ by $\mu(a)$.
  Thus, to the vertex labeled by $W_{1,2}$, we attach two new children labeled by $\langle e_1 \rangle = \{\, w \in W_{1,2} : w \mu(a) \in \langle e_2 \rangle \,\}$ and $\langle e_2 \rangle = \{\, w \in W_{1,2} : w \mu(a) \in \langle e_3 \rangle  \,\}$. (There are two different vertices with label $\langle e_2 \rangle$, one attached to each of the two roots.) At this point the process stops, because every label of a leaf is mapped into a label of a leaf by $\mu(a)$.
\end{example}

\begin{lemma} \label{l:point-count}
  Let $V$ be an irreducible component of $\overline \Omega$.
  Let $\pi\colon V \to W$ be an epimorphism with $\dim W \ge 2$.
  If $Y \subsetneq W$ is a closed subset, then there exists $C \in \bR_{>0}$ such that
  \[
    \card{\{\, a \in \bP(\pi(\Omega) \setminus Y) : h_W(a) \le N \, \}} \ \ge\  C N^{\frac{1}{e\card{X}}} \qquad\text{for $N \in \bZ_{\ge 0},$}
  \]
  where $e = \max\{\dim Y, 1\} + n - m$.
\end{lemma}

\begin{proof}
  Extend $\pi$ to $\pi\colon K^{1\times n} \to W$.
  Let $C > 0$ be such that $h(a\mu(x)) \le h(a) + C$ for all $x \in X$ and $a \in K^{1 \times n} \setminus \ker(\mu(x))$.
  We may moreover assume $h_W(\pi(a)) \le h(a) + C$ for all $a \in K^{1\times n} \setminus \ker(\pi)$.
  If $a \in \Omega_{\le M} \setminus \ker(\pi)$, then $h_W(\pi(a)) \le h(u) + (M+1)C$.
  Choosing $M = (N - h(u)) / C - 1$ we find that all $a \in \Omega_{\le M} \setminus \ker(\pi)$ have $h_W(\pi(a)) \le N$.
  Thus it suffices to show that $\bP(\pi(\Omega_{\le M}) \setminus Y)$ contains at least $C' M^{\frac{1}{e\card{X}}}$ points for some $C' > 0$.

  If $\bP(\pi(\Omega_{\le M}) \setminus Y)$ contains $l$ points, then $\Omega_{\le M} \subseteq \pi^{-1}(Y) \cup \pi^{-1}(P_1) \cup \cdots \cup \pi^{-1}(P_l)$ for some $1$-dimensional vector spaces $P_1$, $\ldots\,$,~$P_l$.
  Since $\dim Y < m$ and $\overline{\pi(\Omega)}=W$, we have $\Omega \not \subseteq \pi^{-1}(Y) \cup \pi^{-1}(P_1) \cup \cdots \cup \pi^{-1}(P_l)$.

  By \cref{l:closed-bound}, for $l \ge \max\{1,k\}$ with $k = \card{\irr(Y)}$,
  \[
    M < (l+k)^{(e-1)\card{X}} + 1 \le (2l)^{e\card{X}}. \qedhere
  \]
\end{proof}

\section{Proof of \ref{main:polya}\texorpdfstring{$\,\Rightarrow\,$}{=>}\ref{main:unam-auto}}
\label{sec:polya-unam-auto}

Having made the necessary preparations, in this section we prove \ref{main:polya}$\,\Rightarrow\,$\ref{main:unam-auto} of \cref{t:main}.
Let $\cP(Q)$ denote the power set of a set $Q$.
The following lemmas very closely parallel the corresponding results on semi-monomial matrices used in the proof of a decomposition theorem for rational functions between free monoids; see \cite[Chapter V.2]{sakarovitch09} and \cref{rem:transducer}.
\begin{lemma} \label{l:unambiguous-automaton}
  Let $\cA=(Q,I,E,T)$ be a weighted automaton on the alphabet $X$ with coefficients in $R$.
  Suppose that there exists $\cS \subseteq \cP(Q)$ such that $Q = \bigcup_{M \in \cS} M$, and all of the following conditions are satisfied.
  \begin{enumerate}
  \item \label{unam:a1} There exists an $M \in \cS$ containing all initial states of $\cA$.
  \item \label{unam:a2} For $M \in \cS$ and $x \in X$, there exists an $N \in \cS$ such that, whenever there is an edge from $p \in M$ to $q \in Q$ labeled by $x$, then $q \in N$.
  \item \label{unam:a3} For every state $q \in Q$, $x \in X$, and $M \in \cS$, there exists at most one state $p \in M$ that has an edge from $p$ to $q$ labeled by $x$.
  \item \label{unam:a4} Every $M \in \cS$ contains at most one terminal state.
  \end{enumerate}
  Then $\cA$ is unambiguous.
\end{lemma}

\begin{proof}
  We need to show that for a word $w=a_1\cdots a_l \in X^*$ with $a_1$, $\ldots\,$,~$a_d \in X$, there exists at most one accepting path in $\cA$ that is labeled by $w$.
  Suppose that there are two accepting paths
  \[
    (p_0,a_1,p_1) (p_1,a_2,p_2) \cdots (p_{l-1},a_l,p_l) \quad\text{and}\quad (q_0,a_1,q_1) (q_1,a_2,q_2) \cdots (q_{l-1},a_l,q_l).
  \]

  We first show that for every $j \in [0,l]$, there exists a set $M_j \in \cS$ with $p_j$,~$q_j \in M_j$.
  For $j=0$, note that $p_0$ and $q_0$ are initial states, hence by \ref*{unam:a1}, there exists $M_0 \in \cS$ with $p_0$,~$q_0 \in M_0$.
  Now, if $j \in [1,l]$ and $p_{j-1}$,~$q_{j-1} \in M_{j-1}$, then \ref*{unam:a2} implies that there exists $M_j \in \cS$ with $p_j$,~$q_j \in M_j$.

  Since the paths are accepting, $p_l$ and $q_l$ are terminal states.
  Since $p_l$, $q_l \in M_l$, condition \ref*{unam:a4} implies $p_l=q_l$.
  If $p_j=q_j$ for some $j \in [1,l]$, then, since we already know $p_{j-1}$,~$q_{j-1} \in M_{j-1}$, condition \ref*{unam:a3} implies $p_{j-1}=q_{j-1}$. Thus, altogether we have $p_j=q_j$ for all $j\in [1,l]$ and hence the two paths are the same.
  Thus we have shown that $\cA$ is unambiguous.
\end{proof}

For the statement of the next lemma we fix the following notation:
Let $e_1$, $\ldots\,$,~$e_n$ denote the standard basis vectors of $R^{1\times n}$.
For $M \subseteq [1,n]$ we set $V(M) = \langle e_\nu : \nu \in M \rangle_K$.
The subscripts $v_i$ and $\mu(x)_{\nu,j}$ below refer to the respective coordinates.

\begin{lemma} \label{l:unambiguous}
  Let $(u, \mu, v)$ be a linear representation of rank $n$ with coefficients in $R$.
  Suppose that there exists $\cS \subseteq \cP([1,n])$ such that $[1,n] = \bigcup_{M \in \cS} M$, and all of the following conditions are satisfied.
  \begin{enumerate}
  \item \label{unam:1} There exists an $M \in \cS$ with $u \in V(M)$.
  \item \label{unam:2} For every $M \in \cS$ and $x \in X$, there exists $N \in \cS$ with $V(M) \mu(x) \subseteq V(N)$.
  \item \label{unam:3} For every $M \in \cS$, $x \in X$, and $j \in [1,n]$, there exists at most one $\nu \in M$ with $\mu(x)_{\nu,j} \ne 0$.
  \item \label{unam:4} For every $M \in \cS$ there exists at most one $\nu \in M$ with $v_\nu \ne 0$.
  \end{enumerate}

  Then the weighted automaton $\cA$ associated to the linear representation $(u, \mu, v)$  is unambiguous.
\end{lemma}

\begin{proof}
  Following the construction of the associated automaton $\cA$, the conditions above translate directly into the ones of \cref{l:unambiguous-automaton}.
\end{proof}

We are now ready to prove the main implication over a field.

\begin{proposition} \label{p:polya-unambig}
  Every rational Pólya series over $K$ is recognized by an unambiguous weighted automaton \textup{(}with weights in $K$\textup{)}.
\end{proposition}

\begin{proof}
  Let $S$ be a rational Pólya series, and let $(u, \mu, v)$ be a minimal linear representation of $S$, chosen as in \cref{l:good-basis}.
  Hence $\Omega \coloneqq u \mu(X^*) \subseteq G_0^{1\times n}$ and $\Omega v \subseteq G_0$.
  If $n=0$, then $S=0$, and $S$ is recognized by the trivial automaton with empty set of states.
  To avoid this corner case, from now on assume $n \ge 1$.
  Then $v = e_1 \in K^{n \times 1}$ is the first standard basis vector by \cref{l:good-basis}.

  We may replace $K$ by the field generated by all coefficients in $u$, $v$, and $\mu(x)$ for $x \in X$.
  Thus, we may without restriction assume that $K$ is finitely generated over its prime field.
  We may also assume that $K$ is infinite; otherwise \cref{p:deterministic} implies the even stronger claim that $S$ is recognized by a deterministic automaton.

  Applying \cref{l:repr-directsum}, we can assume $K^{1\times n} = W_1 \oplus \cdots \oplus W_k$ with $W_i = \langle e_{m_1+\cdots+m_{i-1}+1}, \ldots, e_{m_1+\cdots+m_i} \rangle_K$ and $\irr(\overline \Omega) = \{ W_1, \ldots, W_k \}$.
  Without restriction $u \in W_1$.
  Note that, taking $\Gamma=G_0$ in \cref{l:repr-directsum}, also the properties $\Omega \subseteq G_0^{1\times n}$  and $\Omega v \subseteq G_0$ are preserved by this change of linear representation.

  For $i \in [1,k]$ let
  \[
    M_i = [m_1+\cdots+m_{i-1}+1, m_1+\cdots+m_i],
  \]
  so that $\{\, e_\nu : \nu \in M_i \,\}$ is a basis for $W_i$.
  We show that the linear representation $(u,\mu,v)$, with
  \[
    \cS = \{\, M_i : i \in [1,k] \,\}
  \]
  satisfies the conditions of \cref{l:unambiguous}, from which the claim will follow.
  
  Indeed, \ref{unam:1} holds for $M_1$ since $u \in W_1$.
  Since $\mu(x)$ is continuous, statement \ref{dense:continuous-component} of \cref{l:dense} implies that for every $i \in [1,k]$ and $x \in X$, there exists $j \in [1,k]$ such that $W_i \mu(x) \subseteq W_j$.
  This implies \ref{unam:2}.

  Let $i \in [1,k]$, let $x \in X$, and let $j \in [1,n]$.
  Let $\varphi\colon W_i \to K$ be defined by $\varphi(a) = a \mu(x) e_j^T$.
  By \ref{dense:finitedense} of \cref{l:dense}, the set $\Omega \cap W_i$ is dense in $W_i$.
  Moreover $\Omega \cap W_i \subseteq \sum_{\nu \in M_i} G_0 e_\nu$ since $\Omega \subseteq G_0^{1 \times n}$.
  Finally, if $a \in \Omega$, then $a \mu(x) \in \Omega \subseteq G_0^{1\times n}$, so that $\varphi(\Omega \cap W_i) \subseteq G_0$.
  If $K$ has characteristic $0$, we can thus apply \cref{l:functional} to the vector space $W_i$, its dense subset $\Omega \cap W_i$, and the homomorphism $\varphi \colon W_i \to K$.
  We conclude that there exists at most one $\nu \in M_i$ with $0 \ne \varphi(e_\nu) = e_\nu \mu(x) e_j^T$.
  Since $e_\nu \mu(x) e_j^T$ is the $(\nu,j)$-entry of the matrix $\mu(x)$, there is at most one $\nu \in M_i$ with $\mu(x)_{\nu,j} \ne 0$.
  Thus \ref{unam:3} holds in characteristic $0$.

  If $K$ has positive characteristic, we apply \cref{l:functional-p} instead of \cref{l:functional}; the additional condition in this lemma is satisfied by \cref{l:point-count}.
  This shows \ref{unam:3} in positive characteristic.

  Similarly, applying \cref{l:functional} in characteristic $0$ (respectively \cref{l:functional-p} together with \cref{l:point-count} in positive characteristic) to the map $W_i \to K$ given by $a \mapsto av$, we find that there exists at most one $\nu \in M_i$ with $v_\nu \ne 0$, implying \ref{unam:4}.
\end{proof}

\begin{example} \label{exm:unambig}
  (Continuation of \cref{exm:polya}.)
  We illustrate the construction of the previous proof using the linear representation from \cref{exm:polya}.
  The linear hull decomposes as $\overline \Omega = W_1 \cup W_2$ with $W_1 = \langle e_1+e_2,e_3 \rangle$ and $W_2 = \langle e_1-e_2,e_3 \rangle$.
  Now $W_1 \oplus W_2 \cong K^{1\times 4} = \langle e_1', \ldots, e_4' \rangle$, where we fix the embeddings $W_1 \hookrightarrow K^{1 \times 4}$, $e_1+e_2 \mapsto e_1'$ and $e_3 \mapsto e_2'$, as well as $W_2 \hookrightarrow K^{1\times 4}$, $e_1-e_2 \mapsto e_3'$ and $e_3 \mapsto e_4'$.
  For every $x \in \{a,b,c\}$ we need to choose $f_x\colon \{W_1,W_2\} \to \{W_1,W_2\}$ such that $W_i\mu(x)\subseteq f_x(W_i)$ for $i \in\{1,2\}$. For $a$,~$b$ there is a unique such choice.
Since $W_i\mu(c) \subseteq W_1 \cap W_2$, there are several choices for $c$ and we pick $f_c(W_i)=W_i$ for $i \in \{1,2\}$.
This choice fixes a deterministic automaton describing the transitions between irreducible components (left side of \cref{fig:unambig}).

The newly constructed linear representation on $K^{1\times 4}$ is given by $(u',\mu',v')$ with $u'=(1,1,0,0)$, with $v'=(2,0,0,0)^T$, and with
  \begin{align*}
    \mu'(a) &=
    \left(\! \begin{array}{cc:cc}
      0 & 0 & 2 & 0 \\
      0 & 0 & 0 & 3 \\
      \hdashline
      2 & 0 & 0 & 0 \\
      0 & 3 & 0 & 0
    \end{array}\!\right),&
    \mu'(b) &=
    \left(\! \begin{array}{cc:cc}
      0 & 1 & 0 & 0 \\
      1 & 0 & 0 & 0 \\
      \hdashline
      0 & -1 & 0 & 0 \\
      1 & 0 & 0 & 0
    \end{array}\!\right),&
    \mu'(c) &=
    \left(\! \begin{array}{cc:cc}
      0 & 0 & 0 & 0 \\
      0 & 5 & 0 & 0 \\
      \hdashline
      0 & 0 & 0 & 0 \\
      0 & 0 & 0 & 5
    \end{array} \!\right).
  \end{align*}
  Here the block structure is determined by the choice of transitions between irreducible components, e.g., a different choice of $f_c$ would yield a different matrix $\mu'(c)$.
  The resulting automaton is depicted in the right side of \cref{fig:unambig}.

\begin{figure}
  \centering
  \begin{tikzpicture}[shorten >=1pt,node distance=3.5cm,on grid]
    \node[state]   (s1)                {$1'$};
    \node[state]   (s2) [right of=s1] {$2'$};
    \node[state]   (s3) [below=2.75cm of s1] {$3'$};
    \node[state]   (s4) [right of=s3] {$4'$};
    
    \path[->]
    (s1) edge [bend right=10] node [below] {$b$} (s2)
    (s2) edge [bend right=10] node [above] {$b$} (s1)
    (s3) edge node [right,xshift=5mm,yshift=-2.5mm] {$b$} (s2)
    (s4) edge node [left,xshift=-4mm,yshift=-2.5mm] {$-b$} (s1)

    (s1) edge [bend right=10] node [left,yshift=2mm] {$2a$} (s3)
    (s3) edge [bend right=10] node [right,yshift=2mm] {$2a$} (s1)
    (s4) edge [bend right=10] node [right,yshift=2mm] {$3a$} (s2)
    (s2) edge [bend right=10] node [left,yshift=2mm] {$3a$} (s4)

    (s4) edge [loop right] node {$c$} (s4)
    (s2) edge [loop right] node {$c$} (s2);

    \path[->] +([xshift=-0.15cm,yshift=0.75cm]s1.north) edge ([xshift=-0.15cm]s1.north)
    +([xshift=0.15cm]s1.north) edge ([xshift=0.15cm,yshift=0.75cm]s1.north) node [right,yshift=0.5cm] {$2$};
    \path[->] +([yshift=0.75cm]s2.north) edge (s2.north);
    
    \draw[rounded corners=0.3cm,dotted] (-0.75cm,0.75cm) rectangle ++(6.5cm,-1.5cm);
    \draw[rounded corners=0.3cm,dotted] (-0.75cm,-2cm) rectangle ++(6.5cm,-1.5cm);

    \node at (5.75cm,-0.75cm) [above left,color=darkgray] {$W_1$};
    \node at (5.75cm,-3.5cm) [above left,color=darkgray] {$W_2$};
    
    \node[state]   (Z1) [left=5cm of s1]  {$W_1$};
    \node[state]   (Z2) [left=5cm of s3]  {$W_2$};

    \path[->]
    (Z1) edge [bend right=30] node [left] {$a$} (Z2)
    (Z2) edge [bend right=30] node [right] {$a$} (Z1)

    (Z2) edge node [right] {$b$} (Z1)
    (Z1) edge [loop left] node {$b$} (Z1)
    
    (Z1) edge [loop right] node {$c$} (Z1)
    (Z2) edge [loop right] node {$c$} (Z2);

    \path[->]
    +([xshift=-0.15cm,yshift=0.75cm]Z1.north) edge ([xshift=-0.15cm]Z1.north)
    +([xshift=0.15cm]Z1.north) edge ([xshift=0.15cm,yshift=0.75cm]Z1.north);
  \end{tikzpicture}

  \caption{(\cref{exm:unambig}) \emph{Left:} Our choice of transitions between irreducible components can be depicted as a deterministic automaton. \emph{Right:} An unambiguous weighted automaton recognizing the same series as in \cref{exm:polya,fig:polya}.}
  \label{fig:unambig}
\end{figure}
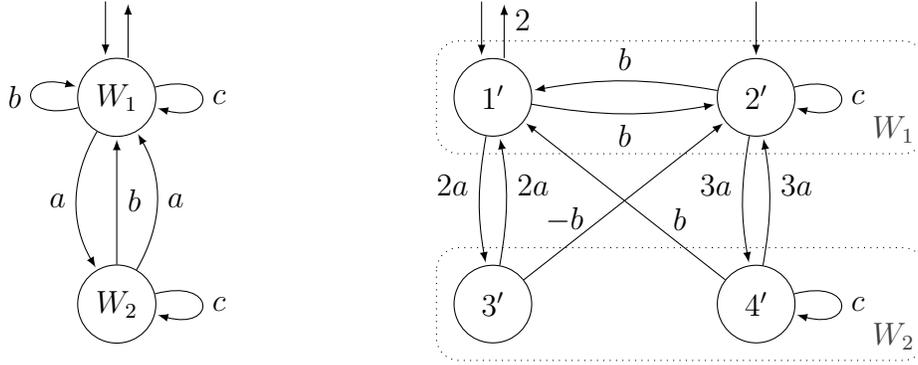
\end{example}

\begin{remark} \label{rem:transducer}
  Let $X$,~$Y$ be finite sets. A \defi{rational function} is a function $f\colon X^* \to Y^*$ whose graph is a rational subset of $X^* \times Y^*$.
  By the \emph{Decomposition Theorem} of Elgot--Mezei every such rational function is a composition of a (pure) sequential function with a (pure) co-sequential function \cite[Chapter V.2]{sakarovitch09} (varying terminology is used, see for instance \cite{arnold-latteux79}; we follow Sakarovitch).

  One of the proofs of this theorem produces, through the use of the Schützenberger covering, a \emph{semi-monomial} linear representation.
  This is the same type of block-matrix structure we have obtained here.
  Consequently we can obtain an analogous decomposition of the weighted automaton:
  every Pólya series is recognized by a weighted finite automaton that is a composition of a sequential function followed by a co-deterministic weighted automaton.
  (A weighted automaton is \defit{co-deterministic} if there is unique final state and for every $x \in X$ and $q \in Q$ there is at most one $p \in Q$ with $E(p,x,q) \ne 0$.)
  As the construction is very similar to the one in \cite[Chapter V.2.2]{sakarovitch09}, we omit the details.
\end{remark}

If $X$ is a singleton, then series recognized by an unambiguous weighted automaton have a particularly simple shape.
In this way we will recover the full univariate result of Pólya, Benzaghou, and Bézivin \cite{polya21,benzaghou70,bezivin87}.

\begin{proposition} \label{p:unambig-ap}
  Suppose $X=\{x\}$ consists of a single element, let $\cA$ be an unambiguous weighted automaton with weights in $R$, and let $S \in R\llangle X \rrangle$ be the series it recognizes.
  Then there exist a finite set $F \subseteq \bZ_{\ge 0}$, an element $d \in \bZ_{\ge 0}$, and for each $r \in [0,d-1]$ elements $a_r \in R$ and $b_r \in R \setminus \{0\}$ such that
  \[
    S(x^{kd+r}) = a_r b_r^k \qquad\text{for all } k \in \bZ_{\ge 0} \text{ and } r \in [0,d-1] \text{ with } xd+r \not \in F.
  \]
\end{proposition}

\begin{proof}
  We may assume that $\cA$ is trim.
  Then, for any two states $p$, $q$ and any $n \ge 0$ there exists at most one path from $p$ to $q$ labeled by $x^n$.

  Let $m$ be the maximal length of an acyclic path in $\cA$, that is, a path that does not visit any vertex twice.
  Then any cycle, that is, a path whose only repeated vertices are the first and the last one, has length bounded by $m+1$.
  Let $d$ be a common multiple of all lengths of cycles in $\cA$, e.g., $d = (m+1)!$.
  Let $r \in \{0,\ldots,d-1\}$ and let us consider the claim for $S(x^{kd+r})$ with $k \ge 0$.
  If $S(x^{kd+r}) =0$ for all but finitely many $d \ge 1$, the claim holds with $a_r=0$.
  Otherwise, let $k_0 \ge 1$ with $k_0d+r > m$ and $S(x^{k_0d+r}) \ne 0$.
  The (unique) accepting path labeled by $x^{k_0d+r}$ must contain a cycle, and so is of the form $pcq$ with $c$ a cycle of length $l$ dividing $d$, and $p$, $q$ paths.
  Let $e \in \bN$ with $d=le$, let $a_r \coloneqq S(x^{pq})$ and let $b_r$ be the product of the weights along $c^e$.
  For all $n \ge 0$, the path $pc(c^e)^nq$ is the unique accepting path for $x^{(k_0+n)d+r}$.
  Hence $S(x^{(k_0+n)d+r}) = a_r b_r^n$ for all $n \ge 0$.
  Thus $S(x^{kd+r}) = (a_rb_r^{-k_0}) b_r^k$ for all $k \ge k_0$.
\end{proof}

\section{Proof of \ref{main:unam-auto}\texorpdfstring{$\,\Leftrightarrow\,$}{<=>}\ref{main:unam-rat}}
\label{sec:unam-auto-unam-rat}

The following proof very closely follows \cite[Proposition 1.3.5]{lothaire02}, where the same result is proved for deterministic automata without weights.
A language $\cL \subseteq X^*$ is a \defi{code} if the elements of $\cL$ are a basis of a free submonoid of $X^*$.

\begin{proposition}[Reutenauer] \label{p:unambig}
  If a rational series $S \in R\llangle X \rrangle$ is recognized by an unambiguous weighted automaton with weights in $R$, then $S$ is unambiguous over $R$.
\end{proposition}

\begin{proof}
  Let $\cA$ be an unambiguous weighted automaton that recognizes $S$.
  We may without restriction assume that $\cA$ is trim.
  Then, for any two states $p$, $q$ and any word $w \in X^*$ there exists at most one path from $p$ to $q$ labeled by $w$.

  For states $p$,~$q \in Q$ and a set $P \subseteq Q$ define
  \[
    S_{p,P,q} = \sum_{\substack{p_1, \ldots, p_{l-1} \in P\\ p=p_0,\, p_l=q\\ a_1,\ldots,a_l \in X\\ l \ge 1}} E(p_0,a_1,p_1) \cdots E(p_{l-1},a_l,p_l)a_1\cdots a_l.
  \]
  In words, the sum is taken over all non-empty paths from $p$ to $q$ with the property that all states strictly in-between are in $P$.
  Since $\cA$ is unambiguous, the words in $\supp(S_{p,P,q})$ are in bijective correspondence with non-empty paths from $p$ to $q$.

  Then
  \[
    S = \sum_{p, q \in Q} I(p)S_{p,Q,q}T(q) + \sum_{p \in Q} I(p)T(p),
  \]
  and the finite sum on the left is unambiguous because $\cA$ is unambiguous.

  It suffices to show that each $S_{p,P,q}$ is unambiguous, and we do so by induction on $\card{P}$.
  If $P = \emptyset$, then $S_{p,P,q}$ is a polynomial and hence unambiguous.
  If $r \not \in P$, then
  \[
    S_{p,P \cup \{r\},q} = S_{p,P,q} + S_{p,P,r} S_{r,P,r}^* S_{r,P,q}.
  \]

  Note that $\supp(S_{r,P,r})$ consist of the words labeling first returns of $r$, that is, non-empty paths starting and ending at $r$ that do not pass through $r$ in-between.
  Using that $\cA$ is unambiguous, it is easily seen that the words in $\supp(S_{r,P,r})$ are a code.
  Hence $S_{r,P,r}^*$ is unambiguous.
  Similarly, we see that the products and the sum are unambiguous, by looking at when a path passes through $r$.
\end{proof}

The converse of the previous implication is also easy to see.
\begin{lemma} \label{l:unambig-series-to-automaton}
  If $S \in R\llangle X \rrangle$ is unambiguous rational, then there exists an unambiguous weighted automaton with weights in $R$ that recognizes $S$.
\end{lemma}

\begin{proof}
  A suitable weighted automaton can inductively be constructed from an unambiguous rational decomposition of $S$.
\end{proof}

\section{Proof of \ref{main:unam-rat}\texorpdfstring{$\,\Rightarrow\,$}{=>}\ref{main:formula}}
\label{sec:unam-rat-formula}

A clever proof of \ref{main:unam-rat}$\,\Rightarrow\,$\ref{main:formula} of \cref{t:main} is given by Reutenauer in the proof of \cite[Proposition 4, (iii)$\,\Rightarrow\,$(ii)]{reutenauer79}.
We opt to give an alternative, somewhat longer but very straightforward, proof of the same result.

\begin{lemma} \label{l:ratexpr}
  Let $A$,~$B \in \bZ\llangle X \rrangle$, and let $\cL$,~$\cK$ be rational languages with $\supp(A) \subseteq \cL$ and $\supp(B) \subseteq \cK$.
  \begin{enumerate}
  \item\label{ratexpr:sum} Suppose $\cK \cap \cL = \emptyset$.
    Then $C= A+B$ is a rational series with $\supp(C) \subseteq \cK \cup \cL$ and
    \[
      C(w) =
      \begin{cases}
        A(w) &\text{if $w \in \cL$,} \\
        B(w) &\text{if $w \in \cK$.}
      \end{cases}
    \]

  \item\label{ratexpr:product} Suppose $\cL \cK$ is unambiguous.
    Then $C=A \charser{\cK} + \charser{\cL} B$ is a rational series with $\supp(C) \subseteq \cL\cK$.
    For $w=uv$ with $u \in \cL$, $v \in \cK$,
    \[
      C(w) = A(u) + B(v).
    \]
  \item\label{ratexpr:star} Suppose that $\cL$ is a code.
    Then
    \[
      C = (1 - \charser{\cL^*} A)\big((\charser{\cL} + A)^* - \charser{\cL^*}\big)
    \]
    is a rational series with $\supp(C) \subseteq \cL^*$.
    For $w=w_1\cdots w_l$ with $w_1$, $\ldots\,$,~$w_l \in \cL$,
    \[
      C(w) = A(w_1) + \dots + A(w_l).
    \]
  \end{enumerate}
\end{lemma}

\begin{proof}
  Throughout, we use that the characteristic series $\charser{\cL} \in \bZ\llangle X\rrangle$ of a rational language $\cL$ is rational.

  \ref*{ratexpr:sum} Clear.

  \ref*{ratexpr:product}
  For $w \in X^*$ we have $A \charser{\cK}(w) = \sum_{w=uv} A(u)\charser{\cK}(v)$.
  A term $A(u)\charser{\cK}(v)$ is nonzero if and only if $u \in \supp(A) \subseteq \cL$ and $v \in \cK$.
  Since $\cL\cK$ is unambiguous there is at most one such term.
  Thus $A \charser{\cK}(w) = A(u)$ if $w \in \cL\cK$ with $w=uv$ where $u \in \cL$, $v \in \cK$, and $A \charser{\cK}(w)=0$ if $w \not \in \cL\cK$.
  An analogous claim holds for $\charser{\cL}{B}$.

  \ref*{ratexpr:star}
  For $w \in \cL^*$ there are uniquely determined $w_1$, $\ldots\,$,~$w_l \in \cL$ with $w=w_1\cdots w_l$.
  Then
  \[
    (\charser{\cL} + A)^*(w) = \sum_{k=0}^l \sum_{1 \le i_1 < \dots < i_k \le l} A(w_{i_1}) \cdots A(w_{i_k}).
  \]
  For $D = (\charser{\cL} + A)^* - \charser{\cL^*}$ we obtain an analogous sum with $k \in [1,l]$.

  Now, $\charser{\cL^*} A(1)=0$ since $1 \not \in \cL$ and, for $j \ge 1$,
  \[
    \charser{\cL^*} A(w_1 \cdots w_j) = \sum_{i=0}^j \charser{\cL^*}(w_1\cdots w_i) A(w_{i+1} \cdots w_j) = A(w_j).
  \]
  Therefore
  \[
    \begin{split}
      \charser{\cL^*} AD(w) &= \sum_{j=1}^l A(w_j)D(w_{j+1}\cdots w_l) =\\
      &= \sum_{j=1}^l A(w_j) \sum_{k=1}^{l-j} \sum_{j+1 \le i_1 < \dots < i_k \le l} A(w_{i_1}) \cdots A(w_{i_k}) \\
      &= \sum_{k=2}^l \sum_{1 \le i_1 < \dots < i_k \le l} A(w_{i_1}) \dots A(w_{i_k}).
    \end{split}
  \]
  Thus $(1-\charser{\cL^*}A)D(w) = \sum_{k=1}^l A(w_k)$.
\end{proof}

A series $a \in \bZ\llangle X\rrangle$ is \emph{linearly bounded} if there exists $C \ge 0$ such that $\abs{a(w)} \le C \length{w}$ for all nonempty words $w \in X^*$.

\begin{proposition} \label{p:unambig-formula}
  Let $S \in R \llangle X \rrangle$ be an unambiguous rational series.
  Then there exist $\lambda_1$, $\ldots\,$,~$\lambda_k \in R \setminus \{0\}$, linearly bounded rational series $a_1$, $\ldots\,$,~$a_k \in \bZ\llangle X^*\rrangle$, and a rational language $\cL$ such that $\supp(a_i) \subseteq \cL$ for all $i \in [1,k]$ and
  \[
    S(w) =
    \begin{cases}
      \lambda_1^{a_1(w)} \cdots \lambda_k^{a_k(w)} & \text{if $w \in \cL$},\\
      0 &\text{if $w \not\in\cL$.}
    \end{cases}
  \]
\end{proposition}

\begin{proof}
  The claim is trivially true if $S$ is a polynomial.
  We show that the property is preserved under unambiguous $+$, $\cdot$, and ${}^*$ constructions.

  Let $S$, $T$ be rational series such that there exist $\lambda_1$, $\ldots\,$, $\lambda_k \in R \setminus \{0\}$, linearly bounded rational series $a_1$, $\ldots\,$,~$a_k$,~$b_1$, $\ldots\,$,~$b_k \in \bZ\llangle X \rrangle$, and rational languages $\cL$, $\cK$ such that $\supp(a_i) \subseteq \cL$, $\supp(b_i) \subseteq \cK$, and
  \[
    S(w) =
    \begin{cases}
      \lambda_1^{a_1(w)} \cdots \lambda_k^{a_k(w)} & \text{if $w \in \cL$},\\
      0 &\text{if $w \not\in\cL$;}
    \end{cases}
  \]
  \[
    T(w) =
    \begin{cases}
      \lambda_1^{b_1(w)} \cdots \lambda_k^{b_k(w)} & \text{if $w \in \cK$},\\
      0 &\text{if $w \not\in\cK$.}
    \end{cases}
  \]
  (We can assume that the $\lambda_i$'s are the same, as we can always extend the set of constants, and set $a_i = 0$, respectively, $b_i = 0$, if $\lambda_i$ does not appear in the expression for $S$, respectively, $T$.)

  We first consider $S + T$ with $\cL \cap \cK = \emptyset$.
  Then $\cL \cup \cK$ is a rational language, $S+T(w)=0$ if $w \not \in \cL \cup \cK$, and 
  \[
    (S+T)(w) = S(w) + T(w) = 
    \begin{cases}
      S(w) = \lambda_1^{a_1(w)} \cdots \lambda_k^{a_k(w)} &\text{if $w \in \cL$},\\
      T(w) = \lambda_1^{b_1(w)} \cdots \lambda_k^{b_k(w)} &\text{if $w \in \cK$}.
    \end{cases}
  \]
  Since $\cL \cap \cK = \emptyset$, we get $(S+T)(w) = \lambda_1^{a_1(w)+b_1(w)} \cdots \lambda_k^{a_k(w)+b_k(w)}$ for all $w \in \cL \cup \cK$.
  Clearly $a_i + b_i$ is a linearly bounded rational series, and $\supp(a_i+b_i) \subseteq \cL \cup \cK$.

  Now consider $ST$ with $\cL \cK$ unambiguous.
  Then $\cL \cK$ is a rational language, and for $w=uv$ with $u \in \cL$, $v \in \cK$,
  \[
    (ST)(w) = S(u)T(v) = \lambda_1^{a_1(u)+b_1(v)} \cdots \lambda_k^{a_k(u) + b_k(v)}.
  \]
  Define series $c_i$ by $c_i(uv) = a_i(u) + b_i(v)$ if $w=uv \in \cL\cK$ with $u \in \cL$, $v \in \cK$, and $c_i(w) =0$ for $w \not \in \cL\cK$.
  Clearly $c_i$ is linearly bounded.
  Since $\cL \cK$ is unambiguous, \cref{l:ratexpr} implies that $c_i$ is rational.

  Now suppose that $\cL = \supp(S)$ is a code and consider $S^*$.
  Then $\cL^*$ is a rational language.
  For $w \in \cL^*$ there exist uniquely determined $w_1$, $\ldots\,$,~$w_l \in \cL$ with $w=w_1\cdots w_l$.
  We have
  \[
    S^*(w) = S^*(w_1\cdots w_l) = S(w_1)\cdots S(w_l) = \lambda_1^{a_1(w_1) + \dots + a_1(w_l)} \cdots \lambda_k^{a_k(w_1) + \dots + a_k(w_l)}.
  \]
  Define
  \[
    c_i(w_1 \cdots w_l) = a_i(w_1) + \cdots + a_i(w_l).
  \]
  and $c_i(w) = 0$ if $w \not \in \cL^*$.
  Then $c_i$ is linearly bounded and, by \cref{l:ratexpr}, again rational.
\end{proof}

\section{Hadamard sub-invertibility}
\label{sec:hadamard}

It is known that every unambiguous rational series is Hadamard sub-invertible, and every Hadamard sub-invertible rational series is a Pólya series \cite[Exercise 3.1 of Chapter 6]{berstel-reutenauer11}.
This is particularly easy for $K=\bQ$.
For arbitrary fields, the same argument works but requires a theorem of Roquette; hence we give the proof in full.

\begin{lemma} \label{l:unambig-hadamard}
  Every unambiguous rational series is Hadamard sub-invertible.
\end{lemma}

\begin{proof}
  Every noncommutative polynomial is Hadamard sub-invertible.
  If $S$,~$T$ are Hadamard sub-invertible, then it is easy to see that the \emph{unambiguous} sums, products, and star operations preserve this property.
\end{proof}

\begin{lemma} \label{l:hadamard-polya}
  Every Hadamard sub-invertible series is a Pólya series.
\end{lemma}

\begin{proof}
  Let $S \in R\llangle X \rrangle \subseteq K\llangle X \rrangle$ be a Hadamard sub-invertible rational series.
  If $\chr K > 0$, let $k$ be the (finite) prime field of $K$; if $\chr K = 0$, let $k = \bZ$.
  It is immediate from the definition of a rational series that there exists a finitely generated $k$-subalgebra $A$ of $K$ containing all coefficients of $S$.
  Since $\sum_{w \in \supp(S)} S(w)^{-1} w$ is also rational, we may moreover assume that $A$ also contains all $S(w)^{-1}$ with $S(w) \ne 0$.
  Hence, the nonzero coefficients of $S$ are contained in $A^\times$.
  The group $A^\times$ is finitely generated by a theorem of Roquette \cite[Corollary 7.5]{lang83} in characteristic $0$, and a slightly easier argument in positive characteristic \cite[Corollary 7.3]{lang83}.
\end{proof}

\section{Putting it all together}
\label{sec:all-together}

The proofs of \cref{t:univariate,t:main} in the case where $R=K$ is a field are now a formality.
To obtain the more general result for domains, we first need to extend \cref{p:polya-unambig} to completely integrally closed domains.

\begin{proposition} \label{p:polya-unambig-general}
  Let $R$ be a completely integrally closed domain.
  Every rational Pólya series in $R\llangle X \rrangle$ is unambiguous rational \textup{(}over $R$\textup{)}.
\end{proposition}

\begin{proof}
  Let $K$ be the quotient field of $R$.
  If $S \in R\llangle X \rrangle$ is a rational Pólya series, \cref{p:polya-unambig} together with \cref{p:unambig} shows that $S$ is unambiguous rational as a series over $K$.
  What remains to be shown is that we can obtain an unambiguous rational decomposition of $S$ in such a way that all the component series have their coefficients in $R$.

  For a domain $A$, let $U_0(A) = A\langle X \rangle$.
  For $k \ge 1$, inductively define $U_k(A)$ as the set of all series obtained as unambiguous sums (that is, having pairwise disjoint support) of series of the form
  \begin{equation}
    \label{eq:expunambig}
    S = \lambda w_0 S_1^* w_1 S_2^* w_2 \cdots S_l^* w_l,
  \end{equation}
  with $l \ge 0$, with $0 \ne \lambda \in A$, with $w_0$, $\ldots\,$,~$w_l \in X^*$, with $S_1$, $\ldots\,$,~$S_l \in \bigcup_{k'=0}^{k-1} U_{k'}(A) \setminus \{0\}$ and satisfying the following conditions:
  \begin{enumerate}
  \item The operations $S_i^*$ are unambiguous, that is $\supp(S_i)$ is a code, for all $i \in [1,l]$, and
  \item the products in \eqref{eq:expunambig} are unambiguous, that is, for every word $w \in \supp(S)$ and every $i \in [1,l]$, there exists a unique word $w_i' \in \supp(S_i^*)$ such that $w=w_0w_1'w_1 \cdots w_l'w_l$.
  \end{enumerate}
  Let $\cU(A)$ denote the set of all unambiguous sums of elements of $\bigcup_{k\ge 0} U_k(A)$.
  By construction, each series in $\cU(A)$ is unambiguous rational.
  Note that $A\langle X \rangle \subseteq \cU(A)$ and that $\cU(A)$ is closed under unambiguous sums and the unambiguous star operation.
  Moreover, by distributivity we see that $\cU(A)$ is closed under unambiguous products.
  Thus $\cU(A)$ is the set of all unambiguous rational series over $A$.

  To conclude the proof of the proposition, we show $U_k(K) \cap R\llangle X \rrangle = U_k(R)$ by induction on $k$.
  For $k=0$ the claim is trivial because $K\langle X \rangle \cap R \llangle X \rrangle = R\langle X \rangle$.
  Suppose now $k \ge 1$ and the claim has been established for $k' < k$.
  Let $S \in U_k(K) \cap R \llangle X \rrangle$.
  Decomposing along unambiguous sums, it suffices to consider $S$ as in \eqref{eq:expunambig}.
  Let $i \in [1,l]$ and let $a_i$ be a nonzero coefficient of $S_i$.
  For each $j \in [1,l] \setminus \{i\}$ pick an arbitrary nonzero coefficient $a_j$ of $S_j$.
  Then $\lambda a_1 \cdots a_{i-1} a_i^m a_{i+1} \cdots a_l$ is a coefficient of $S$ for every $m \ge 0$, and hence contained in $R$.
  Since $R$ is completely integrally closed, we conclude $a_i \in R$.
  Thus $S_i \in U_{k'}(K) \cap R\llangle X \rrangle = U_{k'}(R)$ for some $k' < k$.
  Since $\lambda$ appears as coefficient of $w_0w_1\cdots w_l$ (taking the empty word in each $S_j^*$), we also must have $\lambda \in R$.
  Thus $S \in U_{k}(R)$.
\end{proof}

\begin{proof}[Proof of \cref{t:main}]
  The equivalence \ref{main:unam-auto}$\,\Leftrightarrow\,$\ref{main:unam-rat} is shown in \cref{sec:unam-auto-unam-rat}.
  For $R=K$ a field, the implication \ref{main:polya}$\,\Rightarrow\,$\ref{main:unam-auto} follows by \cref{p:polya-unambig}.
  More generally, for $R$ a completely integrally closed domain, the implication \ref{main:polya}$\,\Rightarrow\,$\ref{main:unam-rat} is shown in \cref{p:polya-unambig-general}.
  Next, the implication \ref{main:unam-rat}$\,\Rightarrow\,$\ref{main:formula} follows from \cref{p:unambig-formula}.
  The implication \ref{main:formula}$\,\Rightarrow\,$\ref{main:polya} is trivial.

  Finally, the implication \ref{main:unam-rat}$\,\Rightarrow\,$\ref{main:hadamard} follows from \cref{l:unambig-hadamard}, and \ref{main:hadamard}$\,\Rightarrow\,$\ref{main:polya} holds by \cref{l:hadamard-polya}.
\end{proof}

\begin{remark} \label{rem:cic}
  Every completely integrally closed domain is integrally closed, and a noetherian domain is completely integrally closed if and only if it is integrally closed.
  Krull domains, and thus in particular factorial domains such as $\bZ$, are completely integrally closed.
  The ring of all algebraic integers is a non-noetherian, completely integrally closed domain.
\end{remark}

To illustrate that some condition needs to be imposed on the domain $R$, the following example gives a Pólya series over an integrally closed, but not completely integrally closed, domain $R$ that is \emph{not} unambiguous rational over $R$.

\begin{example}
  Let $R = \bZ[y^iz : i \ge 1] \subseteq \bZ[y,z]$ and $S = \sum_{i \ge 1} y^izx^i \in R\llbracket x \rrbracket$.
  Then $S$ is a Pólya series, and indeed, over $\bZ[y,z]$ is unambiguous rational as $S=zyx (yx)^*$.

  However, suppose $S$ were unambiguous rational over $R$.
  Then \cref{p:unambig-ap} applies.
  In particular, there exist $d > 0$, $r \ge 0$, and $f$,~$g \in R$ such that $y^{dk+r}z=fg^k$ for every $k \ge 0$.
  However, in $R$ the element $y^iz$ is an irreducible element for each $i \ge 1$, so we must have $g = \pm 1$ and $y^{dk+r}z = \pm f$ for all $k \ge 0$, a contradiction.

  The ring $R$ is not completely integrally closed, because $(yz) y^i \in R$ for all $i \ge 0$, but $y \not \in R$.
  However it is integrally closed:
  Note that $R \subseteq \bZ[yz,y]$ and the latter ring is factorial, hence integrally closed, so that the integral closure of $R$ must be contained in $\bZ[yz,y]$.
  Let $a \in K$ be integral over $R$.
  Then $a \in \bZ[yz,y]$ and hence $a=a' + a''$ with $a' \in \bZ[y]$ and $a'' \in R$.
  Then $a'$ is integral over $R$.
  Hence there exist $m \ge 1$ and $b_0$, $\ldots\,$,~$b_{m-1} \in R$ such that
  \[
    (a')^m + b_{m-1} (a')^{m-1} + \cdots + b_0 = 0.
  \]
  Taking this equation modulo $z$, we see that $a' \in \bZ[y]$ is integral over $\bZ$, forcing  $a' \in \bZ$.
  Hence $a \in R$.
\end{example}

The previous example is no coincidence; more generally the following holds.

\begin{lemma}
  Suppose $R$ is integrally closed but not completely integrally closed.
  Then there exists a rational Pólya series $S \in R\llangle X \rrangle$ such that $S$ is \emph{not} unambiguous rational.
\end{lemma}

\begin{proof}
  Since $R$ is not integrally closed, there exists $0 \ne a \in R$ and $b \in K \setminus R$ such that $ab^i \in R$ for all $i \ge 0$.
  Let $S = \sum_{i \ge 0} ab^ix^i$.
  We show that $S$ is not unambiguous rational.
  Suppose to the contrary that it is.
  Then it is recognized by an unambiguous weighted automaton with weights in $R$, and \cref{p:unambig-ap} shows that there exist $m > 0$, $n \ge 0$ and $c$,~$d \in R \setminus \{0\}$ such that $ab^{mk+n} = cd^k$ for all $k \ge 0$.
  Then $(b^md^{-1})^k=ca^{-1}b^{-n}$ is constant for all $k \ge 0$.
  Substituting $k=0$ and $k=1$, we see $b^md^{-1}=1$.
  Thus $b$ is a root of the polynomial $t^m - d \in R[t]$, hence integral over $R$.
  By hypothesis $b \in R$, a contradiction.
\end{proof}

Finally, we deduce the classical theorem for the univariate theorem as a special case of our result.

\begin{proof}[Proof of \cref{t:univariate}]
  The series $S$ is rational by \cite[Proposition 6.1.1]{berstel-reutenauer11}.
  By \cref{t:main}, there exists an unambiguous weighted automaton $\cA$ on the alphabet $X=\{x\}$ recognizing $S$.
  The claim now follows from \cref{p:unambig-ap}.
\end{proof}

\section{Determinizability}
\label{sec:determinizable}

We finish by showing  \ref{det:det}$\,\Rightarrow\,$\ref{det:dim} of  \cref{t:determinizable}.
The approach in this section is inspired by \cite[Theorem 9]{mohri97}.
Mohri shows that a deterministic weighted automaton over the tropical semiring $(\bR,\max,+)$ has bounded variation; and that an unambiguous weighted automaton with bounded variation is determinizable.

\begin{definition}
  Let $(G,\cdot)$ be a group.
  A map $\ell \colon G \to \bR_{\ge 0}$ is a \emph{length function} if
  \begin{enumerate}
  \item $\ell(1_G) = 0$.
  \item $\ell(gh) \le \ell(g) + \ell(h)$ for all $g$,~$h \in G$.
  \item $\ell(g) = \ell(g^{-1})$ for all $g \in G$.
  \end{enumerate}
\end{definition}

If $K$ has an absolute value $|\cdot|$, and $G \le K^\times$, then $\ell(g) = \abs{\log(\abs{g})}$ defines a length function.
On $(\bZ^r,+)$ we have a length function $(a_1,\ldots,a_r) \mapsto \abs{a_1} + \cdots + \abs{a_r}$.
This induces a length function $\ell$ on any finitely generated free abelian group, since $G/G_{\text{tor}} \cong \bZ^r$.
This length function satisfies
\begin{equation} \label{e:length-finite}
  \tag{\textasteriskcentered}
  \card{\{\, g \in G : \ell(g) \le C \,\}} < \infty \quad \text{for all $C \ge 0$}.
\end{equation}

There is a metric $\sd\colon X^* \times X^* \to \bZ_{\ge 0}$, given by
\[
  \sd(u,v) = \card{u} + \card{v} - 2 \card{\operatorname{lgcd}(u,v)},
\]
where $\operatorname{lgcd}(u,v)$ is the longest common prefix of $u$ and $v$.

\begin{definition}
  Let $G$ be a group and $\ell \colon G \to \bR_{\ge 0}$ a length function.
  A function $f\colon X^* \to G_0$ has \emph{bounded $\ell$-variation} if, for every $c \ge 0$, there exists $C \ge 0$ such that for all $u$, $v$ with $f(u) \ne 0$ and $f(v) \ne 0$,
  \[
    \sd(u,v) \le c \quad\text{implies}\quad \ell(f(u)f(v)^{-1}) \le C.
  \]
\end{definition}

\begin{lemma} \label{l:det-variation}
  Let $\cA$ be a deterministic weighted automaton, and let $S$ be the series recognized by $\cA$.
  Let $G \le K^\times$ be such that
  \begin{itemize}
  \item $S(w) \in G_0$ for $w \in X^*$.
  \item all edge and terminal weights of $\cA$ are contained in $G_0$.
  \end{itemize}
  If $\ell \colon G \to \bR_{\ge 0}$ is a length function, then $S$ has bounded $\ell$-variation.
\end{lemma}

\begin{proof}
  Define
  \[
    C =\max \big\{\, \ell(E(p,a,q)),\, \ell(T(q)) : p,q \in Q, a \in X \text{ with } T(q) \ne 0, E(p,a,q) \ne 0 \,\big\}.
  \]
  Let $w$,~$w' \in X^*$.
  Assume $S(w)\ne 0$ and $S(w') \ne 0$, as otherwise there is nothing to show.

  We may suppose
  \[
    w = u_1 \cdots u_k v_{1} \cdots v_l \quad\text{and}\quad w' = u_1\cdots u_k v_1'\dots v_m',
  \]
  with $u_i$, $v_i$,~$v_i' \in X$ and $\sd(w,w') = l+m$.
  Let $c$ and $c'$ denote the accepting paths labeled by $w$ and $w'$.
  Since $\cA$ is deterministic, we must have
  \begin{align*}
    c &= (p_0,u_1,p_1) \cdots (p_{k-1},u_k,p_k) (p_k,v_1,q_1) \cdots (q_{l-1},v_l,q_l), \quad\text{ and} \\
    c' &= (p_0,u_1,p_1) \cdots (p_{k-1},u_k,p_k) (p_k,v_1',q_1') \cdots (q_{l-1}',v_m,q_m'),
  \end{align*}
  with states $p_i$, $q_i$, $q_i'$.
  For notational convenience, set $p_k=q_0=q_0'$.

  Now
  \[
    \begin{split}
    \ell\bigg(\frac{S(w)}{S(w')}\bigg) &= \ell\bigg(\frac{I(p_0)E(c)T(q_l)}{I(p_0)E(c')T(q_m')}\bigg) \\
    &\le \sum_{i=1}^l \ell(E(q_{i-1},v_i,q_i)) + \sum_{i=1}^m \ell(E(q_{i-1}',v_i',q_i')) + \ell(T(q_l)) + \ell(T(q_m'))\\
    &\le (2+m+l)C = (\sd(w,w') + 2)C.
    \end{split}
  \]
  If $w \ne w'$, then $\sd(w,w') \ge 2$.
  Choosing $C'=C+1$, we have $\ell( S(w)S(w')^{-1} ) \le C'\sd(w,w')$.
\end{proof}

\begin{lemma} \label{l:det-dimbound}
  Let $S$ be a Pólya series.
  Let $\ell \colon G \to \bR_{\ge 0}$ be a length function satisfying \eqref{e:length-finite}.
  Suppose that $S$ has bounded $\ell$-variation.
  If $(u,\mu,v)$ is a minimal linear representation of $S$, its linear hull has dimension at most $1$.
\end{lemma}

\begin{proof}
  Let $w_1$, $\ldots\,$,~$w_n \in X^*$ be such that the vectors $b_i=\mu(w_i)v$ form a basis of $K^{n \times 1}$.
  Let $B$ be the matrix whose $i$-th column is $b_i$.
  Let $\Omega = u \mu(X^*) \subseteq K^{1\times n}$.
  We have to show that $\Omega$ can be covered by finitely many 1-dimensional vector spaces.
  Since $B$ is invertible, it suffices to show the same for $\Omega B = \{\, u \mu(w)B  : w \in X^* \,\}$.

  Since $S$ has bounded $\ell$-variation, there exists $C \ge 0$ such that for every $w \in X^*$ and $i$, $j \in [1,n]$
  \[
    \abs[\Big]{\ell\big( S(ww_i) S(ww_j)^{-1} \big)} \le C,
  \]
  whenever $S(ww_i) \ne 0$ and $S(ww_j) \ne 0$.
  By our assumptions, this implies that the ratio $S(ww_i)S(ww_j)^{-1}$ can only take finitely many values (when $S(ww_j) \ne 0$).
  Noting that $S(ww_i) = u \mu(w) b_i$ is the $i$-th coordinate of $u\mu(w)B$, we conclude that we can cover $\Omega B$ by finitely many $1$-dimensional vector spaces.
\end{proof}

\begin{proof}[{Proof of \cref{t:determinizable}}]
  \ref{det:det}$\,\Rightarrow\,$\ref{det:dim}
  Let $\cA$ be a deterministic weighted automaton that recognizes $S$.
  There exists a finitely generated $G \le K^\times$ such that $S(w) \in G_0$ for all $w \in X^*$.
  This follows from \cref{t:main}, but can also be easily seen directly:
  Since $\cA$ is deterministic, and in particular unambiguous, each coefficient is a product of some weights of the automaton.
  By enlarging $G$ if necessary, we can further assume that all edge and terminal weights are contained in $G_0$.

  Let $\ell \colon G \to \bR_{\ge 0}$ be a length function on $G$ satisfying \eqref{e:length-finite}.
  By \cref{l:det-variation}, the series $S$ has bounded $\ell$-variation.
  \Cref{l:det-dimbound} implies the claim.

  \ref{det:dim}$\,\Rightarrow\,$\ref{det:det}
  By \cref{p:deterministic}.
\end{proof}

\bibliographystyle{hyperalphaabbr}
\bibliography{rational_series}

\end{document}